\documentclass[12pt]{amsart}
\advance\hoffset by -0.5 truecm
\usepackage{amsmath,amscd}
\usepackage{amssymb}
\usepackage{amsthm}
\usepackage{graphicx}

\usepackage{mathrsfs}
\usepackage{hyperref}

\newtheorem{Theorem}{Theorem}[section]

\newtheorem{Lemma}[Theorem]{Lemma}
\newtheorem{Corollary}[Theorem]{Corollary}

\newtheorem{Proposition}[Theorem]{Proposition}

\newtheorem{Definition}[Theorem]{Definition}

\numberwithin{equation}{section}

\def \dim{{\mbox {dim}}\,}

\def\V{\mbox{Var}}

\def\Z{{\mathbb Z}}
\def\mZ{{\mathbb Z}}
\def\R\re
\def\V\mathbb{V}

\def \la{\lambda}

\def \re{{\mathbb R}}
\def \mR{{\mathbb R}}

\def \C{{\mathbb C}}

\def \V{\mathbb{V}}

\def \la{\lambda}

\newcommand{\cjd}{\rangle}
\newcommand{\cjg}{\langle}
\newcommand{\mc}{\mathcal}
\newcommand{\pl}{\partial}
\newcommand{\id}{\mathrm{Id}}

\newcommand{\abs}[1]{\lvert #1 \rvert}
\newcommand{\norm}[1]{\lVert #1 \rVert}

\newcommand{\eps}{\varepsilon}

\def \vd{\overset{\tt{v}}{\nabla}}
\def \hd{\overset{\tt{h}}{\nabla}}
\def \vdiv{\overset{\tt{v}}{\mbox{\rm div}}}
\def \hdiv{\overset{\tt{h}}{\mbox{\rm div}}}

\newcounter{sidenote}
\begin{document}

\title[The X-ray transform for connections]{The X-ray transform for connections in \mbox{negative curvature}} 

\author[C. Guillarmou]{Colin Guillarmou}
\address{DMA, Ecole Normale Superieure, Paris}
\email{cguillar@dma.ens.fr}

\author[G.P. Paternain]{Gabriel P. Paternain}
\address{ Department of Pure Mathematics and Mathematical Statistics,
University of Cambridge,
Cambridge CB3 0WB, UK}
\email {g.p.paternain@dpmms.cam.ac.uk}

\author[M. Salo]{Mikko Salo}
\address{Department of Mathematics and Statistics, University of Jyv\"askyl\"a}
\email{mikko.j.salo@jyu.fi}

\author[G. Uhlmann]{Gunther Uhlmann}
\address{University of Washington / University of Helsinki / Institute for Advanced Study, Hong Kong University of Science and Technology}
\email{gunther@math.washington.edu}




\begin{abstract}
We consider integral geometry inverse problems for unitary connections and skew-Hermitian Higgs fields on manifolds with negative sectional curvature. The results apply to manifolds in any dimension, with or without boundary, and also in the presence of trapped geodesics. In the boundary case, we show injectivity of the attenuated ray transform on tensor fields with values in a Hermitian bundle (i.e.\ vector valued case). We also show that a connection and Higgs field on a Hermitian bundle 
are determined up to gauge by the knowledge of the parallel transport between boundary points along all possible geodesics. 
The main tools are an energy identity, the Pestov identity with a unitary connection, which is presented in a general form, and a precise analysis of the singularities of solutions of transport equations when there are trapped geodesics. In the case of closed manifolds, we obtain similar results modulo the obstruction given by twisted conformal Killing tensors, and we also study this obstruction.
\end{abstract}

\maketitle

\tableofcontents

\section{Introduction} \label{sec_introduction}

There has been considerable activity recently in the study of integral geometry problems on Riemannian manifolds. Part of the motivation comes from nonlinear inverse problems such as boundary rigidity (inverse kinematic problem), scattering and lens rigidity, or spectral rigidity. It turns out that in many cases, there is an underlying linear inverse problem that is related to inverting a geodesic ray transform, i.e.\ to determining a function or a tensor field from its integrals over geodesics. We refer to the survey \cite{PSU4} for some of the recent developments in this direction.

One of the main approaches for studying geodesic ray transforms is based on energy estimates, often coming in the form of a Pestov identity. This approach originates in \cite{Mu} and has been developed by several authors, see for instance \cite{PS, Sh, PSU4}. A simple derivation of the basic Pestov identity in two dimensions was given in \cite{PSU1}. There it was also observed that the Pestov identity may become even more powerful when a suitable connection is included. This fact was used in \cite{PSU1} to establish solenoidal injectivity of the geodesic ray transform on tensors of any order on compact simple surfaces, and it was also used earlier in \cite{PSU2} to study the attenuated ray transform with connection and Higgs field on compact simple surfaces.

The results of \cite{PSU2, PSU1} were restricted to two-dimensional manifolds. In the preprint \cite{PSU_hd} much of the technology was extended to manifolds of any dimension, including a version of the Pestov identity which looks very similar to the two-dimensional one in \cite{PSU1}. However, the arguments of \cite{PSU_hd} do not consider the case of connections.

The main aim of this paper is to generalize the setup of \cite{PSU_hd} to the case where connections and Higgs fields are present. We will state a version of the Pestov identity with a unitary connection that is valid in any dimension $d \geq 2$ (similar identities have appeared before, see \cite{Sha_connection,V}). This will have several applications in integral geometry problems. We will mostly work on manifolds with negative sectional curvature, which will be sufficient for the integral geometry results. In the boundary case, we also invoke the microlocal methods of \cite{G14b, DG14} that allow to treat negatively curved manifolds with trapped geodesics. In this paper we do not employ the new local method introduced in \cite{UhlmannVasy}, which might be effective in the boundary case when $d \geq 3$ if the method could be adapted to the present setting.

\subsection{Main results in the boundary case}
 
Let $(M,g)$ be a compact connected oriented Riemannian manifold with smooth boundary and with dimension $\dim(M) = d \geq 2$. In this paper we will consider manifolds $(M,g)$ with strictly convex boundary, meaning that the second fundamental form of $\partial M \subset M$ is positive definite.
Let $SM = \{ (x,v) \in TM \,;\, \abs{v} = 1 \}$ be the unit sphere bundle with boundary $\partial(SM)$ and projection $\pi:SM\to M$, and write 
$$
\partial_{\pm}(SM) = \{ (x,v) \in \partial(SM) \,;\, \mp \langle v, \nu \rangle > 0 \}
$$
where $\nu$ is the the inner unit normal vector. Note that the sign convention for $\nu$ and $\partial_\pm(SM)$ are opposite to \cite{PSU_hd}.

We denote by $\varphi_t$ the geodesic flow on $SM$ and by $X$ the geodesic vector field on $SM$, so that $X$ acts on smooth functions on $SM$ by 
$$
Xu(x,v) = \frac{\partial}{\partial t} u(\varphi_t(x,v))\Big|_{t=0}.
$$
If $(x,v) \in SM$ denote by $\ell_+(x,v)\in [0,\infty]$ the first time when the geodesic starting at $(x,v)$ exits $M$ in forward time (we write $\ell_+=\infty$ if the geodesic does not exit $M$). We will also write $\ell_-(x,v):=-\ell_+(x,-v)\leq 0$ for the exit time in backward time. We define the incoming ($-$) and outgoing ($+$) tails 
\[ \Gamma_\mp:=\{ (x,v)\in SM \,;\, \ell_\pm(x,v)=\pm \infty\}.\]
When the curvature of $g$ is negative, the set $\Gamma_+\cup \Gamma_-$ has zero Liouville measure (see section \ref{Geosetup}), and similarly $\Gamma_\pm \cap \pl(SM)$ has zero measure for any measure of Lebesgue type on 
$\pl(SM)$.

We recall certain classes of manifolds that often appear in integral geometry problems. A compact manifold $(M,g)$ with strictly convex boundary is called 
\begin{itemize}
\item 
\emph{simple} if it is simply connected and has no conjugate points, and 
\item 
\emph{nontrapping} if $\Gamma_+\cup \Gamma_-=\emptyset$. 
\end{itemize}
For compact simply connected manifolds with strictly convex boundary, we have 
\[
\text{negative sectional curvature} \implies \text{simple} \implies \text{nontrapping}.
\]
Also, any compact nontrapping manifold with strictly convex boundary is contractible and hence simply connected (see \cite[Proposition 2.4]{PSU1}).

In this paper we will deal with negatively curved manifolds that are not necessarily simply connected and may have trapped geodesics. We briefly give an example in which all our results are new and non-trivial.  We consider a piece of a catenoid, that is, a surface $M=S^{1}\times [-1,1]$ with coordinates $(u,v)$ and metric $ds^2=\cosh^{2} v(du^{2}+dv^{2})$, see Figure \ref{fi1}.

\begin{figure}[h]
\includegraphics[scale=.6]{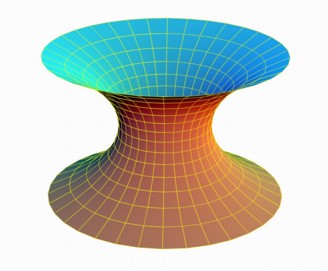}
\caption{A catenoid}
\label{fi1}
\end{figure}

 It is an elementary exercise to check that the boundary is strictly convex and that the surface has negative curvature.  The equations for the geodesics are easily computed: there is a first integral (Clairaut's integral) given
by $\dot{u}\cosh^{2}v=c$ and a second equation of the form $\ddot{v}=\tanh v(\dot{u}^{2}-\dot{v}^{2})$. The curves $t\mapsto (\pm t,0)$
are trapped unit speed closed geodesics  and the union of the tails $\Gamma_{+}\cup \Gamma_{-}$ is determined by the equations $\dot{u}\cosh^2 v=\pm 1$.

\vspace{12pt}

\noindent {\bf X-ray transform.} 
Let $(M,g)$ be a compact manifold with strictly convex boundary, and denote by $M^\circ$ its interior. Given a function $f \in C^{\infty}(SM)$, the \emph{geodesic ray transform} of $f$ is the function $If$ defined by 
$$
If(x,v) = \int_0^{\ell_+(x,v)} f(\varphi_t(x,v)) \,dt, \qquad (x,v) \in \pl_-(SM)\setminus \Gamma_-.
$$
Thus $If$ encodes the integrals of $f$ over all non-trapped geodesics going from $\partial M$ into $M$. 
By \cite[Proposition 4.4]{G14b} (for the existence) and \cite[Lemma 3.3]{G14b} (for the uniqueness), 
when the curvature is negative, there is a unique solution $u\in L^1(SM)\cap C^\infty(SM\setminus \Gamma_-)$ to the transport equation 
$$
Xu = -f \text{ in the distribution sense in }SM^\circ, \quad u|_{\pl_+(SM)} = 0,
$$
and one can define $If$ by 
$$
If = u|_{\pl_-(SM)\setminus \Gamma_-}.
$$

It is not possible to recover a general function $f \in C^{\infty}(SM)$ from the knowledge of $If$. However, in many applications one is interested in the special case where $f$ arises from a symmetric $m$-tensor field on $M$. To discuss this situation it is convenient to consider spherical harmonics expansions in the $v$ variable. For more details on the following facts see \cite{GK2,DS, PSU_hd}. Given any $x \in M$ one can identify $S_x M$ with the sphere $S^{d-1}$. The decomposition 
$$
L^2(S^{d-1}) = \bigoplus_{m=0}^{\infty} H_m(S^{d-1}),
$$
where $H_m(S^{d-1})$ consists of the spherical harmonics of degree $m$, gives rise to a spherical harmonics expansion on $S_x M$. Varying $x$, we obtain an orthogonal decomposition 
$$
L^2(SM) = \bigoplus_{m=0}^{\infty} H_m(SM)
$$
and correspondingly any $f \in L^2(SM)$ has an orthogonal decomposition 
$$
f = \sum_{m=0}^{\infty} f_m.
$$
 We say that a function $f$ has \emph{degree $m$} if $f_k = 0$ for $k \geq m+1$ in this decomposition, 
and we say that $f$ has \emph{finite degree} if it has degree $m$ for some finite $m$. We understand that any $f$ having degree $-1$ is identically zero. 

Solenoidal injectivity of the X-ray transform  can be stated as follows.
\begin{quote}
If $f$ has degree $m$ and $If = 0$, then $f = -Xu$ for some smooth $u$ with degree $m-1$ and $u|_{\partial(SM)} = 0$.
\end{quote}
This has been proved in a number of cases, including the following:
\begin{itemize}
\item 
compact simple manifolds if $m=0$ \cite{Mu} or $m=1$ \cite{AR};
\item 
compact simple manifolds with non-positive curvature if $m \geq 2$ \cite{PS};
\item 
compact simple manifolds if $d=2$ and $m \geq 2$ \cite{PSU1};
\item 
generic compact simple manifolds if $d \geq 3$ and $m = 2$ \cite{SU3};
\item 
manifolds foliated by convex hypersurfaces if $d \geq 3$ and $m \leq 2$ \cite{UhlmannVasy, SUV_tensor};
\item 
compact manifolds with strictly convex boundary and non-positive curvature \cite{G14b}.
\end{itemize}

\vspace{12pt}

\noindent {\bf Attenuated ray transform.} 
Next we discuss the attenuated geodesic ray transform involving a connection and Higgs field. For motivation and further details, we refer to Section \ref{section_attenuated_motivation} and \cite{PSU2,P2}.

Let $(M,g)$ be a compact negatively curved manifold with strictly convex boundary. We will work with vector valued functions and systems of transport equations, and for that purpose it is convenient to use the framework of Hermitian vector bundles. Let $\mc{E}$ be a Hermitian vector bundle over $M$, and let $\nabla$ be a connection on $\mc{E}$. We assume that $\nabla$ is \emph{unitary} (or \emph{Hermitian}), meaning that 
\begin{equation} \label{Hermitian}
Y\cjg u,u'\cjd_{\mc{E}}=\cjg \nabla_Y u,u'\cjd_{\mc{E}}+ \cjg u,\nabla_Y u'\cjd_{\mc{E}}
\end{equation}
for all vector fields $Y$ on $M$ and sections $u,u'\in C^\infty(M;\mc{E})$. Both denominations, unitary and Hermitian, are of common use in the literature and here we will use them indistinctively. Let also $\Phi$ be a \emph{skew-Hermitian Higgs field}, i.e.\ a smooth section $\Phi: M \to {\rm End}_{\rm sk}(\mc{E})$ where ${\rm End}_{\rm sk}(\mc{E})$ is the bundle of skew-Hermitian endomorphisms on $\mc{E}$.

If $SM$ is the unit sphere bundle of $M$, the natural projection $\pi: SM\to M$ gives rise to the pullback bundle $\pi^* \mc{E}$ and pullback connection $\pi^* \nabla$ over $SM$. For convenience we will omit $\pi^*$ and denote the lifted objects by the same letters as downstairs (thus for instance we write $C^{\infty}(M ; \mc{E})$ for the sections of the original bundle $\mc{E}$ over $M$, and $C^{\infty}(SM ; \mc{E})$ for the sections of $\pi^* \mc{E}$). As in the case of functions, we can decompose the space of $L^2$ sections as 
$L^2(SM;\mc{E})=\bigoplus_{m=0}^{\infty} H_m(SM; \mc{E})$; see Section \ref{sec_pestov}.

The geodesic vector field $X$ can be viewed as acting on sections of $\mc{E}$ by 
\begin{equation} \label{defbfX}
 \mathbb{X}u := \nabla_Xu , \quad u\in C^\infty(SM; \mc{E}).
 \end{equation}
If $f \in C^{\infty}(SM; \mc{E})$, the \emph{attenuated ray transform} of $f$ is defined by 
\begin{equation}
I_{\nabla,\Phi} f = u|_{\pl_-(SM)\setminus \Gamma_-}
\label{eq:defat}
\end{equation}
where $u \in L^1(SM; \mc{E})\cap C^\infty(SM\setminus \Gamma_- ;\mc{E})$ is the unique solution of the transport equation (here $M^\circ$ is the interior of $M$) 
$$
(\mathbb{X}+\Phi)u = -f \text{ in the distribution sense in } SM^\circ, \quad u|_{\pl_+(SM)} = 0.
$$
We refer to Proposition \ref{Rpm0} for the proof of the existence and uniqueness of solution. The following theorem proves solenoidal injectivity of the attenuated ray transform (with attenuation given by any unitary connection and skew-Hermitian Higgs field) on any negatively curved manifold with strictly convex boundary.

\begin{Theorem} \label{mainthm_boundary1}
Let $(M,g)$ be a compact manifold with strictly convex boundary and negative sectional curvature, let  $\mc{E}$ be a Hermitian bundle over $M$, and let $\nabla$ be a unitary connection and $\Phi$ a skew-Hermitian Higgs field on $\mc{E}$. If $f \in C^{\infty}(SM;\mc{E})$ has degree $m$ and if the attenuated ray transform of $f$ vanishes (meaning that $I_{\nabla,\Phi} f = 0$), then there exists $u \in C^{\infty}(SM;\mc{E})$ which has degree $m-1$ and 
satisfies 
\[
f = -(\mathbb{X}+\Phi)u, \qquad u|_{\partial(SM)} = 0,
\]
where $\mathbb{X}$ is defined by \eqref{defbfX}.
\end{Theorem}

Note in particular that for $m=0$, the above theorem states that any $f \in C^{\infty}(M;\mc{E})$ with $I_{\nabla,\Phi} f = 0$ must be identically zero. The conclusion of Theorem \ref{mainthm_boundary1} is also known for compact simple two-dimensional manifolds (follows by combining the methods of \cite{PSU2} and \cite{PSU1}; this result even for magnetic geodesics may be found in \cite{Ainsworth_paper}). We will use the assumption of strictly negative curvature to deal with large connections and Higgs fields in any dimension.

\vspace{12pt}

\noindent {\bf Parallel transport between boundary points: the X-ray transform for connections and Higgs fields.}
We now discuss a related nonlinear inverse problem, where one tries to determine a connection and Higgs field  on a Hermitian bundle $\mc{E}$ in $(M,g)$ from parallel transport between boundary points. 
This problem largely motivates the present paper; for more details see \cite{PSU2}. Given a compact negatively curved manifold $(M,g)$ with strictly convex boundary, the scattering relation 
\[S_g: \partial_-(SM)\setminus \Gamma_- \to \partial_+(SM)\setminus \Gamma_+,\quad 
 (x,v) \mapsto \varphi_{\ell_+(x,v)}(x,v)\] 
 maps the start point and direction of a geodesic to the end point and direction. 
 If $\mc{E}$ is a Hermitian bundle, $\nabla$ is a unitary connection and $\Phi$ a skew-Hermitian Higgs field, we consider the parallel transport  with respect to $(\nabla,\Phi)$, which is the smooth bundle map 
 $T^{\nabla,\Phi}: \mc{E}|_{\pl_-(SM)\setminus \Gamma_-}\to \mc{E}|_{\pl_+(SM)\setminus \Gamma_+}$ defined by
$T^{\nabla,\Phi}(x,v;e):= F(S_g(x,v))$ where $F$ is a section of $\mc{E}$ over the geodesic 
$\gamma(t)=\pi(\varphi_t(x,v))$ satisfying the ODE
\[ \nabla_{\dot{\gamma}(t)}F(\gamma(t))+\Phi(\gamma(t)) F(\gamma(t))=0, 
\quad F(\gamma(0))=e.\]
The following theorem shows that on compact manifolds with negative curvature and strictly convex boundary, the parallel transport between boundary points determines the pair $(\nabla,\Phi)$ up to the natural gauge equivalence.

\begin{Theorem} \label{mainthm_boundary2}
Let $(M,g)$ be a compact manifold of negative sectional curvature with strictly convex boundary, and 
let $\mc{E}$ be a Hermitian bundle on $M$.
Let $\nabla$ and $\tilde{\nabla}$ be two unitary connections on $\mc{E}$ 
and let $\Phi$ and $\tilde{\Phi}$ be two skew-Hermitian Higgs fields. 
If  the parallel transports agree, i.e. $T^{\nabla,\Phi}=T^{\tilde{\nabla},\tilde{\Phi}}$, 
then there is a smooth section $Q:M \to \mathrm{End}(\mc{E})$ with values in unitary endomorphisms such that $Q|_{\partial M}={\rm Id}$ and 
$\tilde{\nabla}=Q^{-1}\nabla Q$, $\tilde{\Phi}=Q^{-1}\Phi Q$. \end{Theorem}

The map $(\nabla,\Phi)\mapsto T^{\nabla,\Phi}$ is sometimes called the non-abelian Radon transform, or the X-ray transform for a non-abelian connection and Higgs field.

Theorem \ref{mainthm_boundary2} was proved for compact simple surfaces (not necessarily negatively curved) in \cite{PSU2}, and for certain simple manifolds if the connections are $C^1$ close to another connection with small curvature in \cite{Sha_connection}. For domains in the Euclidean plane the theorem was proved in \cite{FU}
assuming that the connections have small curvature and in \cite{E} in general.
For connections which are not compactly supported (but with suitable decay conditions at infinity), \cite{No} establishes local uniqueness of the trivial connection and gives examples in which global uniqueness fails.
The examples are based on a connection between 
the Bogomolny equation in Minkowski $(2+1)$-space and the scattering data $T^{\nabla,\Phi}$ considered above. As it is explained in \cite{W}
(see also \cite[Section 8.2.1]{D}), certain soliton solutions $(\nabla,\Phi)$
have the property that when restricted to space-like planes the scattering data is trivial. In this way one obtains connections in $\mathbb R^2$ with the property
of having trivial scattering data but which are not gauge equivalent to the trivial connection. Of course these pairs are not compactly supported in $\mathbb R^2$ but they have a suitable decay at infinity.


\subsection{Main results in the closed case}

Let now $(M,g)$ be a closed oriented Riemannian manifold of dimension $\dim(M) = d \geq 2$. The geodesic ray transform of a function $f \in C^{\infty}(SM)$ is the function $If$ given by 
$$
If(\gamma) = \int_0^{L(\gamma)} f(\varphi_t(x,v)) \,dt, \qquad \gamma \in \mathcal{G},
$$
where $\mathcal{G}$ is the set of periodic unit speed geodesics on $M$ and $L(\gamma)$ is the length of $\gamma$. Of course it makes sense to consider situations where $(M,g)$ has many periodic geodesics. A standard such setting is the case where $(M,g)$ is Anosov, i.e.\ the geodesic flow of $(M,g)$ is an Anosov flow on $SM$, meaning that there is a continuous flow-invariant splitting 
$$
T(SM) = E_0 \oplus E_s \oplus E_u
$$
where $E_0$ is the flow direction and the stable and unstable bundles $E_s$ and $E_u$ satisfy for all $t > 0$ 
\begin{equation}\label{stable/unstable}
\norm{D \varphi_t|_{E_s}} \leq C \rho^t, \quad  \norm{D \varphi_{-t}|_{E_u}} \leq C \eta^{-t}
\end{equation}
with $C > 0$ and $0 < \rho < 1 < \eta$. Closed manifolds with negative sectional curvature are Anosov \cite{KH}, but there exist Anosov manifolds with large sets of positive curvature \cite{Ebe} and Anosov surfaces embedded in $\mR^3$ \cite{DP}. Anosov manifolds have no conjugate points \cite{K, A, Man} but may have focal points \cite{Gul}.

If $(M,g)$ is closed Anosov and if $f \in C^{\infty}(SM)$ satisfies $If = 0$, the smooth Livsic theorem \cite{dLMM} implies that $Xu = -f$ for some $u \in C^{\infty}(SM)$. The tensor tomography problem for Anosov manifolds can then be stated as follows:

\begin{quote}
Let $(M,g)$ be a closed Anosov manifold. If $f$ has degree $m$ and if $Xu = -f$ for some smooth $u$, show that $u$ has degree $m-1$.
\end{quote}

We wish to consider the same problem where a connection and Higgs field are present. Let $\mc{E}$ be a Hermitian bundle, $\nabla$ be a unitary connection on $\mc{E}$ and $\Phi$ a skew-Hermitian Higgs field. 
Using the decomposition $L^2(SM;\mc{E})=\bigoplus_{m=0}^{\infty} H_m(SM; \mc{E})$ as before,
the operator $\mathbb{X}=\nabla_X$ acts on  $\Omega_m = H_m(SM;\mc{E}) \cap C^{\infty}(SM;\mc{E})$  by 
$$
\mathbb{X} = \mathbb{X}_+ + \mathbb{X}_-
$$
where $\mathbb{X}_{\pm}: \Omega_m \to \Omega_{m \pm 1}$ (see Section \ref{sec_pestov}). The operator $\mathbb{X}_+$ is overdetermined elliptic, and $\mathbb{X}_-$ is of divergence type.

There is a possible obstruction for injectivity of the attenuated ray transform: if $u \in \mathrm{Ker}(\mathbb{X}_+) \cap \Omega_{m+1}$ and $u \neq 0$, then setting $f = -\mathbb{X}_- u$ we have $\mathbb{X}u = -f$ where $f$ has degree $m$ but $u$ has degree $m+1$. Thus the analogue of Theorem \ref{mainthm_boundary1} for closed manifolds can only hold if $\mathrm{Ker}(\mathbb{X}_+)$ is trivial. We call elements in the kernel of $\mathbb{X}_+|_{\Omega_m}$  \emph{twisted Conformal Killing Tensors} (CKTs in short) of degree $m$. We say that there are no nontrivial twisted CKTs if $\mathrm{Ker}(\mathbb{X}_+|_{\Omega_m}) = \{0\}$ for all $m \geq 1$. The dimension of $\mathrm{Ker}(\mathbb{X}_+|_{\Omega_m})$ is a conformal invariant (see Section \ref{sec_pestov}). In the case of the trivial line bundle with flat connection, twisted CKTs coincide with the usual CKTs, and these cannot exist on any manifold whose conformal class contains a metric with negative sectional curvature or a rank one metric with nonpositive sectional curvature \cite{DS, PSU_hd}.

The following result proves solenoidal injectivity of the attenuated ray transform on closed negatively curved manifolds with no nontrivial twisted CKTs, and also gives a substitute finite degree result if twisted CKTs exist.

\begin{Theorem} \label{main_thm_closed_finitedegree}
Let $(M,g)$ be a closed manifold with negative sectional curvature, let $\mc{E}$ be a Hermitian bundle, 
and let $\nabla$ be a unitary connection and $\Phi$ a skew-Hermitian Higgs field on $\mc{E}$. 
If $f \in C^{\infty}(SM;\mc{E})$ has finite degree, and if $u \in C^{\infty}(SM;\mc{E})$ solves the equation 
$$
(\mathbb{X}+\Phi)u = -f \text{ in } SM,
$$
then $u$ has finite degree. If in addition there are no twisted CKTs,
and $f$ has degree $m$, then $u$ has degree $\max\{m-1,0\}$ 
(and $u \in \mathrm{Ker}(\mathbb{X}_+|_{\Omega_0})$ if $m=0$).
\end{Theorem}

We conclude with a few results on twisted CKTs. The situation is quite simple on manifolds with boundary: any twisted CKT that vanishes on part of the boundary must be identically zero. The next theorem extends \cite{DS} which considered the case of a trivial line bundle with flat connection. This result will be used as a component in the proof of Theorem \ref{mainthm_boundary1} (for $\Gamma=\partial M$ and $\pi^{-1}\Gamma=\partial(SM)$).

\begin{Theorem} \label{main_thm_twistedckts_boundary}
Let $(M,g)$ be a compact Riemannian manifold, let $\mc{E}$ be a Hermitian bundle, and let $\nabla$ be a unitary connection on $\mc{E}$. If $\Gamma$ is a hypersurface of $M$ and for some $m \geq 0$ one has 
$$
\mathbb{X}_+ u = 0 \text{ in }SM, \ \ u \in \Omega_m, \ \ u|_{\pi^{-1}\Gamma} = 0,
$$
then $u = 0$.
\end{Theorem}

We next discuss the case of closed two-dimensional manifolds. If $(M,g)$ is a closed Riemannian surface with genus $0$ or $1$, then nontrivial CKTs exist even for the flat connection on the trivial line bundle (consider conformal Killing vector fields on the sphere or flat torus). The next result considers surfaces with genus $\geq 2$, and gives a condition for the connection ensuring the absence of nontrivial twisted CKTs. The proof is based on a Carleman estimate.

To state the condition, note that if $\mc{E}$ is a Hermitian vector bundle of rank $n$ and $\nabla$ is a unitary connection on $\mc{E}$, then the curvature $f^{\mc{E}}$ of $\nabla$ is a $2$-form with values in skew-Hermitian endomorphisms of $\mc{E}$. In a trivializing neighborhood $U \subset M$, $\nabla$ may be represented as $d+A$ where $A$ is an $n \times n$ matrix of $1$-forms, and the curvature is represented as $dA + A \wedge A$, an $n \times n$ matrix of $2$-forms. If $d=2$ and if $\star$ is the Hodge star operator, then $i \star f^{\mc{E}}$ is a smooth section on $M$ with values in Hermitian endomorphisms of $\mc{E}$, and it has real eigenvalues $\lambda_1 \leq \ldots \leq\lambda_n$ counted with multiplicity. Each $\lambda_j$ is a Lipschitz continuous function $M \to \mR$. Below $\chi(M)$ is the Euler characteristic of $M$.

\begin{Theorem} \label{main_thm_twistedckts_twodim}
Let $(M,g)$ be a closed Riemannian surface, let $\mc{E}$ be a Hermitian vector bundle of rank $n$ over $M$, and let $\nabla$ be a unitary connection on $\mc{E}$. Denote by $\lambda_1 \leq \ldots \leq \lambda_n$ the eigenvalues of $i \star f^{\mc{E}}$ counted with multiplicity. If $m \geq 1$ and if 
$$
2\pi m \chi(M) < \int_M \lambda_1 \,dV \ \ \text{ and } \quad \int_M \lambda_n \,dV < -2\pi m \chi(M),
$$
then any $u \in \Omega_m$ satisfying $\mathbb{X}_+ u = 0$ must be identically zero.
\end{Theorem}

The conditions for $\lambda_1$ and $\lambda_n$ are conformally invariant (they only depend on the complex structure on $M$) and sharp: \cite{Pa2} gives examples of connections on a negatively curved surface for which $\lambda_1= K$ (the Gaussian curvature), so one has $\int_M \lambda_1 \,dV
 = 2\pi \chi(M)$, and these connections admit twisted CKTs of degree $1$. Further examples of nontrivial twisted CKTs on closed negatively curved surfaces are in \cite{P1,P2}.
 
For closed manifolds of dimension $d \geq 3$, our results on absence of twisted CKTs are less precise but we have the following theorem.

\begin{Theorem} \label{thm_main_ckt_hd}
Let $(M,g)$ be a closed manifold whose conformal class contains a negatively curved manifold, let $\mc{E}$ be a Hermitian vector bundle over $M$, and let $\nabla$ be a unitary connection. There is $m_0 \geq 1$ such that $\mathrm{Ker}(\mathbb{X}_+|_{\Omega_m}) = \{0\}$ when $m \geq m_0$ (one can take $m_0=1$ if $\nabla$ has sufficiently small curvature) .
\end{Theorem}

We also obtain a result regarding transparent pairs, that is, connections and Higgs fields for which the parallel transport along periodic geodesics coincides with the parallel transport for the flat connection. This closed manifold analogue of Theorem \ref{mainthm_boundary2} is discussed in Section \ref{sec_transparent_pairs}.

\vspace{12pt}

\noindent{\bf Open questions.}
Here are some open questions related to the topics of this paper:
\begin{itemize}
\item 
Does Theorem \ref{mainthm_boundary1} hold for compact simple manifolds when $d \geq 3$, or for manifolds satisfying the foliation condition in \cite{UhlmannVasy}? The result is known for compact simple two-dimensional manifolds \cite{PSU2,PSU1,Ainsworth_paper}.
\item 
Does Theorem \ref{main_thm_closed_finitedegree} hold for closed Anosov manifolds? This is known if $d=2$ and one has the flat connection on a trivial bundle \cite{DS0,PSU3,G14a}.
\item 
Do the results above remain true for general connections and Higgs fields (not necessarily unitary or skew-Hermitian)? If $d=2$ this is known for line bundles (see \cite{PSU1}) and domains in $\mR^2$ \cite{E}. Another partial result for $d=2$ is in \cite{Ainsworth_thesis}.
\item 
Can one find other conditions for the absence of nontrivial twisted CKTs on closed manifolds when $d \geq 3$ besides Theorem \ref{thm_main_ckt_hd}? Is this a generic property?
\end{itemize}

\vspace{12pt}

\noindent {\bf Structure of the paper.}
This paper is organized as follows. Section \ref{sec_introduction} is the introduction and states the main results. In Section \ref{section_attenuated_motivation} we explain the relation between attenuated ray transforms and connections, and include some preliminaries regarding connections on vector bundles. Section \ref{sec_pestov} proves the Pestov identity with a connection, introduces operators relevant to this identity, and discusses spherical harmonics expansions and related estimates. In Section \ref{sec_finite_degree} we use the Pestov identity to prove the finite degree part of Theorem \ref{main_thm_closed_finitedegree} (both in the boundary and closed case). Section \ref{sec_twisted_ckts_raytransform} begins the study of twisted CKTs, proves Theorem \ref{main_thm_closed_finitedegree} in full and also proves Theorem \ref{main_thm_twistedckts_boundary}. Section \ref{sec_regularity} finishes the proof of Theorem \ref{mainthm_boundary1} using regularity results obtained via the microlocal approach of \cite{G14b}. Section \ref{scatteringequivalent} proves the scattering data result (Theorem \ref{mainthm_boundary2}), Section \ref{sec_twistedckts_twodim} discusses twisted CKTs in two dimensions and proves Theorem \ref{main_thm_twistedckts_twodim}, and the final Section \ref{sec_transparent_pairs} discusses transparent pairs and a simplified analogue of Theorem \ref{mainthm_boundary2} for closed manifolds.

\vspace{12pt}

\noindent {\bf Acknowledgements.}
C.G.\ was partially supported by grants ANR-13-BS01-0007-01 and ANR-13-JS01-0006. M.S.\ was supported in part by the Academy of Finland (Centre of Excellence in Inverse Problems Research) and an ERC Starting Grant (grant agreement no 307023). G.U.\ was partly supported by NSF and a Simons Fellowship.

\section{Attenuated ray transform and connections} \label{section_attenuated_motivation}

In this section we motivate briefly how connections may appear in integral geometry, and collect basic facts about connections on vector bundles (see \cite{Jost} for details). Readers who are familiar with these concepts may proceed directly to Section \ref{sec_pestov}.

\vspace{12pt}

\noindent {\bf Euclidean case.} We first consider the closed unit ball $M = \{ x \in \mR^d \,;\, \abs{x} \leq 1 \}$ with Euclidean metric. If $f \in C^{\infty}(M)$, the \emph{attenuated X-ray transform} $I_{\mathcal{A}} f$ of $f$ is defined by 
\[
I_{\mathcal{A}}f(x,v) = \int_0^{\ell_+(x,v)} f(x+tv) \exp \left[ \int_0^t \mathcal{A}(x+sv,v) \,ds \right]\,dt, \qquad (x,v) \in \partial(SM),
\]
where $SM = M \times S^{d-1}$ is the unit sphere bundle, $\partial(SM) = \partial M \times S^{d-1}$ is its boundary, $\ell_+(x,v)$ is the time when the line segment starting from $x$ in direction $v$ exits $M$, and $\mathcal{A} \in C^{\infty}(SM)$ is the attenuation coefficient. If $\mathcal{A} = 0$, then $I_{\mathcal{A}}$ is the classical X-ray transform which underlies the medical imaging methods CT and PET. The attenuated transform arises in various applications, such as the medical imaging method SPECT \cite{finch} or the Calder\'on problem \cite{DKSaU}, and often $\mathcal{A}$ has simple dependence on $v$. We will consider attenuations of the form 
\[
\mathcal{A}(x,v) = \sum_{j=1}^d A_j(x) v^j + \Phi(x)
\]
where $A_j, \Phi \in C^{\infty}(M)$.

Define the function 
\[
u(x,v) := \int_0^{\ell_+(x,v)} f(x+tv) \exp \left[ \int_0^t \mathcal{A}(x+sv,v) \,ds \right]\,dt, \qquad (x,v) \in SM.
\]
Then clearly $If = u|_{\partial(SM)}$. A computation shows that $u$ satisfies the first order differential equation (\emph{transport equation}) 
\begin{equation} \label{intro_transport_equation}
Xu(x,v) + \left[ \sum_{j=1}^d A_j(x)v^j + \Phi(x) \right] u(x,v) = -f(x), \qquad (x,v) \in SM,
\end{equation}
where $X$ is the geodesic vector field acting on functions $w \in C^{\infty}(SM)$ by $Xw(x,v) = \frac{\partial}{\partial t} w(x+tv,v) |_{t=0}$. The inverse problem of recovering $f$ from $I_{\mathcal{A}}f$ can thus be reduced to finding the source term $f$ in \eqref{intro_transport_equation} from boundary values of the solution $u$.

We now give a geometric interpretation of the transport equation \eqref{intro_transport_equation}. Complex valued functions $f \in C^{\infty}(M)$ may be identified with sections of the trivial vector bundle $\mc{E} := M \times \C$. The complex $1$-form $A := \sum_{j=1}^d A_j(x) \,dx^j$ on $M$ gives rise to a \emph{connection} $\nabla = d + A$ on $\mc{E}$, taking sections of $\mc{E}$ to $1$-form valued sections via 
\[
\nabla f = d_M f + Af, \qquad f \in C^{\infty}(M).
\]
The projection $\pi: SM \to M$ induces a pullback bundle $\pi^* \mc{E}$ and pullback connection $\pi^* \nabla$ over $SM$. Since $\mc{E}$ is the trivial line bundle, one has $\pi^* \mc{E} = SM \times \C$, sections of $\pi^* \mc{E}$ can be identified with functions in $C^{\infty}(SM)$, and $\pi^* \nabla$ is given by 
\[
(\pi^* \nabla) u = d_{SM} u + (\pi^* A) u, \qquad u \in C^{\infty}(SM).
\]
The geodesic vector field $X$ is a vector field on $SM$, and induces a map $\mathbb{X} := (\pi^* \nabla)_X$ on sections of $\pi^* \mc{E}$ given by 
\[
\mathbb{X} u(x,v) = Xu(x,v) + [\sum_{j=1}^d A_j(x)v^j] u(x,v), \qquad u \in C^{\infty}(SM).
\]
The transport equation \eqref{intro_transport_equation} then becomes 
\begin{equation} \label{intro_transport_equation_2}
\mathbb{X} u + (\pi^* \Phi) u = -\pi^* f
\end{equation}
where $u$ and $\pi^* f$ are now sections of $\pi^* \mc{E}$, and $\Phi$ is a smooth section from $M$ to the bundle of endomorphisms on $\mc{E}$ (\emph{Higgs field}).

\vspace{12pt}

\noindent {\bf Hermitian bundles.} 
The above discussion extends to more general vector bundles over manifolds. Let $(M,g)$ be a compact oriented Riemannian manifold with or without boundary, having dimension $d = \dim(M)$. Let $\mc{E}$ be a Hermitian vector bundle over $M$ having rank $n \geq 1$, i.e.\ each fiber $\mc{E}_x$ is an $n$-dimensional complex vector space equipped with a Hermitian inner product $\langle \,\cdot\, , \,\cdot\, \rangle$ varying smoothly with respect to base point.


We assume that $\mc{E}$ is equipped with a connection $\nabla$, so for any vector field $Y$ in $M$ there is a $\C$-linear map on sections 
\[
\nabla_Y: C^{\infty}(M ; \mc{E}) \to C^{\infty}(M ; \mc{E})
\]
which satisfies $\nabla_{fY} u = f \nabla_Y u$ and $\nabla_Y (fu) = (Yf)u + f \nabla_Y u$ for $f \in C^{\infty}(M)$ and $u \in C^{\infty}(M ; \mc{E})$. There is a corresponding map 
\[
\nabla: C^{\infty}(M ; \mc{E}) \to C^{\infty}(M ; T^*M \otimes \mc{E}).
\]
If $\mc{E}$ is trivial over a coordinate neighborhood $U \subset M$ and if $( e_1, \ldots, e_n )$ is an orthonormal frame for local sections over $U$, then $\nabla$ has the local representation 
\begin{align*}
\nabla(\sum_{k=1}^nu^k e_k) &= \sum_{k=1}^n(du^k + \sum_{l=1}^n A^k_l u^l) \otimes e_k, \\
\nabla_Y(\sum_{k=1}^nu^k e_k) &= \sum_{k=1}^n(du^k + \sum_{l=1}^n A^k_l u^l)(Y) e_k \notag
\end{align*}
where $A = (A^k_l)$ is an $n \times n$ matrix of $1$-forms in $U$, called the connection $1$-form corresponding to $( e_1, \ldots, e_n )$. Locally one writes 
\[
\nabla = d + A.
\]
We say that $\nabla$ is a \emph{unitary} connection (or \emph{Hermitian} connection) if it is compatible with the Hermitian structure:
\[
Y \langle u, u' \rangle = \langle \nabla_Y u, u' \rangle + \langle u, \nabla_Y u' \rangle, \qquad Y \text{ vector field, } u,u' \in C^{\infty}(M ; \mc{E}).
\]
Equivalently, $\nabla$ is Hermitian if in any trivializing neighbourhood the matrix $(A^k_l)$ is skew-Hermitian.

If $\nabla$ is a connection on a complex vector bundle $\mc{E}$, we can define a linear operator 
\[
\nabla: C^{\infty}(M ; \Lambda^k(T^*M) \otimes \mc{E}) \to C^{\infty}(M ; \Lambda^{k+1}(T^*M) \otimes \mc{E}) 
\]
for $k \geq 1$ by requiring that 
\[
\nabla(\omega \wedge u) = d\omega \wedge u + (-1)^k \omega \wedge \nabla u, \quad \omega \in C^{\infty}(M ; \Lambda^k(T^*M)), \ \ u \in C^{\infty}(M ; \mc{E})
\]
where $\omega \wedge u$ is a natural wedge product of a differential form $\omega$ and a section $u$. The \emph{curvature} of $(\mc{E},\nabla)$ is the operator 
\[
f^{\mc{E}} = \nabla \circ \nabla: C^{\infty}(M ; \mc{E}) \to C^{\infty}(M ; \Lambda^2(T^*M) \otimes \mc{E}).
\]
This is $C^{\infty}(M)$-linear and can be interpreted as an element of $C^{\infty}(M ; \Lambda^2(T^*M) \otimes \mathrm{End}(\mc{E}))$, where $\mathrm{End}(\mc{E})$ is the bundle of endomorphisms of $\mc{E}$. If $\mc{E}$ is trivial over $U$ and $\nabla = d+A$ with respect to an orthonormal frame $( e_1, \ldots, e_n )$ for local sections over $U$, then 
\[
f^{\mc{E}}(\sum_{k=1}^nu^k e_k) = \sum_{k,l=1}^n(dA^k_l + A^k_m \wedge A^m_l) u^l \otimes e_k.
\]
Locally one writes 
\[
f^{\mc{E}} = dA + A \wedge A.
\]
If $\mc{E}$ and $\nabla$ are unitary, then $dA + A \wedge A$ is a skew-Hermitian matrix of $2$-forms and 
$f^{\mc{E}} \in C^{\infty}(M ; \Lambda^2(T^*M) \otimes \mathrm{End_{sk}}(\mc{E}))$.

\vspace{10pt}

\noindent {\bf Pullback bundles.}\ Next we consider the lift of $\nabla$ to the pullback bundle over $SM$. Let $\pi: SM \to M$ be the natural projection. The pullback bundle of $\mc{E}$ by $\pi$ is 
\[
\pi^* \mc{E} = \{ ((x,v), e) \,;\, (x,v) \in SM, \ e \in \mc{E}_x \}.
\]
Then $\pi^* \mc{E}$ is a Hermitian bundle over $SM$ having rank $n$. The connection $\nabla$ induces a pullback connection $\pi^* \nabla$ in $\pi^* \mc{E}$, defined uniquely by 
\[
(\pi^* \nabla)(\pi^* u) = \pi^* (\nabla u), \qquad u \in C^{\infty}(M ; \mc{E}).
\]
In coordinates $\pi^* \nabla$ looks as follows: if $U$ is a trivializing neighbourhood of $\mc{E}$ and if $( e_1, \ldots, e_n )$ is an orthonormal frame of sections over $U$, then $( \tilde{e}_1, \ldots, \tilde{e}_n )$ where $\tilde{e}_j = e_j \circ \pi$ is a frame of sections of $\pi^* \mc{E}$ over $SU$, and 
\[
(\pi^* \nabla)(\sum_{k=1}^nu^k \tilde{e}_k) = \sum_{k=1}^n (d_{SM} u^k + \sum_{l=1}^n(\pi^* A^k_l) u^l) \otimes \tilde{e}_k, \qquad u^k \in C^{\infty}(SU).
\]
Later we will omit $\pi^*$ and we will denote the pullback bundle and connection just by $\mc{E}$ and $\nabla$ (we will also  write $e_j$ instead of $\tilde{e}_j = e_j \circ \pi$).

\section{Pestov identity with a connection} \label{sec_pestov}

In this section we will state and prove the Pestov identity with a connection. We will also give several related inequalities that will be useful for proving the main results.

\vspace{12pt}

\noindent 3.1. {\bf Unit sphere bundle.}
To begin, we need to recall certain notions related to the geometry of the unit sphere bundle. We follow the setup and notation of \cite{PSU_hd}; for other approaches and background information see \cite{GK2,Sh,Pa,Kn,DS}.

Let $(M,g)$ be a $d$-dimensional compact Riemannian manifold with or without boundary, having unit sphere bundle $\pi: SM\to M$, and let $X$ be the geodesic vector field. We equip $SM$ with the Sasaki metric. If $\mathcal V$ denotes the vertical subbundle
given by $\mathcal V=\mbox{\rm Ker}\,d\pi$, then there is an orthogonal splitting with respect to the Sasaki metric:
\begin{equation}\label{TSM}
TSM=\re X\oplus {\mathcal H}\oplus {\mathcal V}.
\end{equation}
The subbundle ${\mathcal H}$ is called the horizontal subbundle. Elements in $\mathcal H(x,v)$ and $\mathcal V(x,v)$ are canonically identified with elements in the codimension one subspace $\{v\}^{\perp}\subset T_{x}M$ by the isomorphisms
\[ d\pi_{x,v} : \mc{V}(x,v)\to \{v\}^{\perp} ,  \quad \mc{K}_{x,v}: \mathcal H(x,v)\to \{v\}^{\perp},\]
here $\mc{K}_{(x,v)}$ is the connection map coming from Levi-Civita connection.
We will use these identifications freely below.  

We shall denote by $\mathcal Z$ the set of smooth functions $Z:SM\to TM$ such that $Z(x,v)\in T_{x}M$ and $\langle Z(x,v),v\rangle=0$ for all $(x,v)\in SM$.
Another way to describe the elements of $\mathcal Z$ is a follows. Consider the pull-back bundle $\pi^*TM$ over $SM$.  Let $N$ denote the subbundle of $\pi^*TM$ whose fiber over $(x,v)$
is given by $N_{(x,v)}=\{v\}^{\perp}$. Then $\mathcal Z$ coincides with the smooth sections
of the bundle $N$. Notice that $N$ carries a natural scalar product and thus an $L^{2}$-inner product 
(using the Liouville measure on $SM$ for integration).

Given a smooth function $u\in C^{\infty}(SM)$ we can consider its gradient $\nabla u$ with respect to the Sasaki metric. 
Using the splitting above we may write uniquely in the decomposition \eqref{TSM}
\[\nabla u=((Xu)X,\hd u,  \vd u). \]
The derivatives $\hd u\in  \mc{Z}$ and $\vd u\in \mc{Z}$ are called horizontal and vertical derivatives respectively. Note that this differs from the definitions in \cite{Kn,Sh} since here all objects are defined on $SM$ as opposed to $TM$.

Observe that $X$ acts on $\mathcal Z$ as follows:
\begin{equation}\label{XonZ}
XZ(x,v):=\frac{DZ(\varphi_{t}(x,v))}{dt}|_{t=0}
\end{equation}
where $D/dt$ is the covariant derivative with respect to Levi-Civita connection and $\varphi_t$ is the geodesic flow.  With respect to the $L^2$-product on $N$, the formal adjoints of $\vd:C^{\infty}(SM)\to\mathcal Z$ and $\hd:C^{\infty}(SM) \to \mathcal Z$ are denoted by $-\vdiv$ and $-\hdiv$ respectively. Note that since $X$ leaves invariant the volume form of the Sasaki metric we have $X^*=-X$ for both actions of $X$ on $C^{\infty}(SM)$ and $\mathcal Z$.
In what follows, we will need to work with the complexified version of $N$ with its natural inherited Hermitian product. This will be clear from the context and we shall employ the same letter $N$ to denote the complexified
bundle and also $\mathcal Z$ for its sections.

\vspace{12pt}

\noindent 3.2. {\bf Hermitian bundles.}
Consider now a Hermitian vector bundle $\mc{E}$ of rank $n$ over $M$ with a Hermitian product $\langle \,\cdot\,,\,\cdot\,\rangle_{\mc{E}}$, and let $\nabla^{\mc{E}}$ be a Hermitian connection on $\mc{E}$ (i.e. satisfying\ \eqref{Hermitian}). Using the projection $\pi: SM \to M$, we have the pullback bundle $\pi^* \mc{E}$ over $SM$ and pullback connection $\pi^* \nabla^{\mc{E}}$ on $\pi^* \mc{E}$. For convenience, we will omit $\pi^*$ and use the same notation $\mc{E}$ and $\nabla^{\mc{E}}$ also for the pullback bundle and connection.

\medskip

\noindent{\bf Remark on notation}. The reader may have noticed that we have adorned the notation for the unitary connection with the superscript $\mc{E}$. The reason for doing so at various points in what is about to follow is to make sure that there is a clear signal of the influence of the connection in the vertical and horizontal components.
We hope this will not cause confusion.

\medskip

If $u\in C^\infty(SM; \mc{E})$, then $\nabla^{\mc{E}}u\in C^\infty(SM; T^*(SM)\otimes \mc{E})$, and using the Sasaki metric on $T(SM)$ we can identify this with an element of 
$C^\infty(SM; T(SM)\otimes \mc{E})$, and thus we can split according to \eqref{TSM}
\[ \nabla^{\mc{E}}u = (\mathbb{X}u, \hd\,  ^\mc{E} u,  \vd \,^\mc{E} u) ,\quad \mathbb{X}u:= \nabla_X^{\mc{E}}u\]
and we can view $\hd \, ^\mc{E} u$ and $\vd \, ^\mc{E} u$ as elements in $C^\infty(SM; N\otimes \mc{E})$. The operator $\mathbb{X}$ acts on $C^\infty(SM; \mc{E})$ and we can also define a similar operator, still denoted 
by $\mathbb{X}$, on $C^\infty(SM;N\otimes \mc{E})$ by 
\begin{equation}\label{bfXonE}
\mathbb{X}(Z\otimes e):= (XZ) \otimes e + Z \otimes (\mathbb{X}e),   \qquad Z \otimes e \in C^\infty(SM;N\otimes \mc{E})
\end{equation}
where $X$ acts on $\mc{Z}$ by \eqref{XonZ}. There is a natural Hermitian product $\langle \,\cdot\,,\,\cdot\, \rangle_{N\otimes \,\mc{E}}$ on $N\otimes \,\mc{E}$ induced by $g$ and 
$\langle \,\cdot\,,\,\cdot\, \rangle_{\mc{E}}$. We define $\hdiv\, ^{\mc{E}}$ and $\vdiv\, ^{\mc{E}}$ to be
the adjoints of $-\hd\, ^{\mc{E}}$ and $-\vd\, ^{\mc{E}}$ in the $L^2$ inner product.

Next we define curvature operators. If $R$ is the Riemann curvature tensor of $(M,g)$, we can view it as an operator on the bundles $N$  and $N\otimes \mc{E}$ over $SM$ by the actions 
\begin{equation}\label{curvatureg}
R(x,v)w := R_x(w,v)v ,   \quad  R(x,v) (w\otimes e):= (R_x(w,v)v)\otimes e
\end{equation} 
if $(x,v)\in SM$, $w\in \{v\}^\perp$, and $e\in \mc{E}_{(x,v)}$. The curvature of the connection $\nabla^{\mc{E}}$ on $\mc{E}$ is denoted $f^\mc{E}\in C^\infty(M;\Lambda^2T^*M \otimes {\rm End}_{\rm sk}(\mc{E}))$ and it is a $2$-form with values in skew-Hermitian 
endomorphisms of $\mc{E}$. In particular, to $f^{\mc{E}}$ we can associate an operator 
$F^{\mc{E}}\in C^\infty(SM; N\otimes {\rm End}_{\rm sk}(\mc{E}))$ defined by 
\begin{equation}\label{defFE} 
\cjg f^{\mc{E}}_{x}(v,w)e,e'\cjd_{\mc{E}} =\cjg F^{\mc{E}}(x,v)e,w \otimes e' \cjd_{N\otimes \,\mc{E}}, 
\end{equation}
where $(x,v)\in SM, \ w\in \{v\}^{\perp}, \ e,e'\in \mc{E}_{(x,v)}.$

Next we give a technical lemma which expresses $F^{\mc{E}}$ in terms of the local connection 
$1$-form $A\in C^\infty(U; T^*M\otimes {\rm End}_{\rm sk}(\mathbb{C}^n))$ of $\nabla^{\mc{E}}$ in a local orthonormal frame $(e_1,\dots,e_n)$ of $\mc{E}$ over a chart $U\subset M$.
In that basis, the curvature $f^{\mc{E}}$ can be written as the $2$-form $f^{\mc{E}}=dA+A\wedge A$. 
We pull back everything to $SM$ (including the frame) and also view $A$ as an element of 
$C^\infty(SU; {\rm End}_{\rm sk}(\mathbb{C}^n))$ by setting $A(x,v):=A_x(v)$. 
\begin{Lemma} \label{FEvsA}
In the local orthonormal frame $(e_1,\dots,e_n)$, the expression of $F^{\mc{E}}$ in terms of the connection $1$-form 
$A$ is 
\[ F^{\mc{E}}=  X(\vd A)-\hd A +[A,\vd A] \]
as elements of $C^{\infty}(SU ; N \otimes \mathrm{End}_{\mathrm{sk}}(\C^n))$.
\end{Lemma}

\begin{proof}
Note that we can interpret the claim as an identity for $n \times n$ matrix functions with entries in $N$, where $X, \vd, \hd$ act elementwise. Write $A=(A^k_l)_{k,l=1}^n$, where each $A^k_l$ is a scalar 1-form. Since $v\mapsto A^k_l(x,v)$ is linear, $\langle \vd A^k_l(x,v),w\rangle=A^k_l(x,w)$. From this we easily derive
\[g([A,\vd A](x,v),w)=A(x,v)A(x,w)-A(x,w)A(x,v)=(A\wedge A)_{x}(v,w) \]
where $g(\,\cdot\,, w)$ acts elementwise. We just need to prove that 
\[
g(X(\vd A)(x,v)-\hd A(x,v),w)=(dA)_{x}(v,w).
\]
It suffices to check this equality when $A$ is a scalar 1-form. 

Let $e_{v}(t)$ denote the parallel transport of $v$ along the geodesic $\gamma_{w}(t)$ determined by $(x,w)$.
Similarly, let $e_{w}(t)$ denote the parallel transport of $w$ along the geodesic $\gamma_{v}(t)$ determined by $(x,v)$. By definition of $dA$:
\[(dA)_{x}(v,w)=\frac{d}{dt}|_{t=0}A_{\gamma_{v}(t)}(e_{w}(t))-\frac{d}{dt}|_{t=0}A_{\gamma_{w}(t)}(e_{v}(t)) .\]
But by definition of $\hd$ we have
\[\frac{d}{dt}|_{t=0}A_{\gamma_{w}(t)}(e_{v}(t))=g(\hd A(x,v),w)\]
since the curve $t\mapsto (\gamma_{w}(t),e_{v}(t))\in SM$ goes through $(x,v)$ and its tangent vector has only horizontal component equal to $w$.
Finally
\begin{align*}
\langle X\vd A(x,v),w\rangle&=\langle \frac{D}{dt}|_{t=0}\vd A(\gamma_{v}(t),\dot{\gamma}_{v}(t)),e_{w}(0)\rangle\\
&=\frac{d}{dt}|_{t=0}\langle \vd A(\gamma_{v}(t),\dot{\gamma}_{v}(t)),e_{w}(t)\rangle\\
&=\frac{d}{dt}|_{t=0}A_{\gamma_{v}(t)}(e_{w}(t))\\
\end{align*}
and the lemma is proved.
\end{proof}

\vspace{12pt}

\noindent 3.3. {\bf Pestov identity with a connection.}
We begin with some basic commutator formulas, which generalize the corresponding formulas in \cite[Lemma 2.1]{PSU_hd} to the case of where one has a Hermitian bundle with unitary connection. The proof also gives local frame representations for the operators involved (this could be combined with \cite[Appendix A]{PSU_hd} to obtain local coordinate formulas)

\begin{Lemma} \label{lemma_basic_commutator_formulas}
The following commutator formulas hold on $C^{\infty}(SM;\mc{E})$:
\begin{align}
[\mathbb{X},\vd \, ^{\mc E}]&=-\hd\, ^{\mc{E}}, \label{eq:commXV}\\
[\mathbb{X},\hd\, ^\mc{E}]&=R\,\vd\, ^{\mc{E}}+F^{\mc{E}},  \label{eq:commXH}\\
\hdiv\,^{\mc{E}}\,\vd\,^{\mc{E}}-\vdiv\,^{\mc{E}}\hd\,^{\mc{E}}&=(d-1)\mathbb{X}, \label{eq:commdiv}
\end{align}
where the maps $R$ and $F^{\mc{E}}$ are defined in \eqref{curvatureg} and \eqref{defFE}.
Taking adjoints, we also have the following commutator formulas on 
$C^{\infty}(SM, N\otimes \mc{E})$:
\begin{align*}
[\mathbb{X},\vdiv\,^{\mc{E}}] &= -\hdiv\,^{\mc{E}}, \\
[\mathbb{X},\hdiv\,^{\mc{E}}] &= -\vdiv\, ^{\mc{E}} R+(F^{\mc{E}})^*
\end{align*}
where $(F^{\mc{E}})^*:C^{\infty}(SM;N\otimes \mc{E})\to C^{\infty}(SM;\mc{E})$ is the 
$L^2$-adjoint of $C^{\infty}(SM;\mc{E})\ni u\mapsto F^{\mc{E}}u\in C^{\infty}(SM,N\otimes\mc{E})$. 
\end{Lemma}
\begin{proof}
It suffices to prove these formulas for a local orthonormal frame $(e_1,\dots,e_n)$ of $\mc{E}$ over a trivializing neighborhood  $U\subset M$.
The connection in this frame will be written as $d+A$ for some connection $1$-form $A\in C^{\infty}(U;T^*U\otimes {\rm End}_{\rm sk}(\C^n))$, i.e.\ we have (using the Einstein summation convention with sums from $1$ to $n$)
\[
\nabla^{\mc{E}}(\sum_{k=1}^n u^k e_k) = \sum_{k=1}^n (du^k + \sum_{l=1}^nA^k_l u^l) \otimes e_k.
\]
We alternatively view $A$ as an element in $C^{\infty}(SU;{\rm End}_{\rm sk}(\C^n))$ of degree $1$ in the variable $v$, by setting $A(x,v) = A_x(v)$.
We pull back the frame to $SU$ (and continue to write $(e_j)$ for the frame) and the connection. Then we get the local frame representations 
\begin{align*}
\mathbb{X}(\sum_{k=1}^nu^k e_k) &= \sum_{k=1}^n (Xu^k + \sum_{l=1}^n A^k_l u^l) e_k, \\
\vd\,^{\mc{E}}(\sum_{k=1}^n u^k e_k) &=  \sum_{k=1}^n (\vd u^k) \otimes e_k.
\end{align*}
Note in particular that $\vd \, ^{\mc{E}}$ does not depend on the connection. To compute a local representation for the horizontal derivative,  we take an orthonormal frame $(Z_1,\dots,Z_{d-1})$ of $N$ over 
$SU$. We can write 
\[
\hd\,^{\mc{E}} (\sum_{k=1}^nu^k e_k) = \sum_{j=1}^{d-1}\sum_{k=1}^n (du^k + \sum_{l=1}^nA^k_l u^l)(Z_j) Z_j \otimes e_k.
\]
Since any $1$-form $a$ satisfies $\vd (a(x,v)) = \sum_{j=1}^{d-1} a(Z_j) Z_j$, we obtain 
\[
\hd\,^{\mc{E}} (\sum_{k=1}^n u^k e_k) = \sum_{k=1}^n(\hd u^k + \sum_{l=1}^n(\vd A^k_l) u^l) \otimes e_k.
\]

We can now use the above formulas and \eqref{bfXonE} to compute 
\[
[\mathbb{X}, \vd\,^{\mc{E}}](\sum_{k=1}^n u^k e_k) = \sum_{k=1}^n([X,\vd] u^k - \sum_{l=1}^n(\vd A^k_l) u^l) \otimes e_k.
\]
By \cite[Lemma 2.1]{PSU_hd} we have $[X, \vd]u^k = -\hd u^k$ and thus the first identity \eqref{eq:commXV} is proved. We also get 
\[\begin{split}
[\mathbb{X}, \hd\,^{\mc{E}}](\sum_{k=1}^n u^k e_k) = & \sum_{k=1}^n( [X, \hd] u^k + \sum_{l=1}^n(X \vd A^k_l - \hd A^k_l) u^l ) \otimes e_k\\
&  + \sum_{k,l,r=1}^n((A^k_r \vd A^r_l - (\vd A^k_r) A^r_l) u^l) \otimes e_k
\end{split}\]
which proves the second identity \eqref{eq:commXH} by using the fact that $[X, \hd]u^k=R\vd u^k$ (\cite[Lemma 2.1]{PSU_hd}) and Lemma \ref{FEvsA} which expresses $F^{\mc{E}}$ in terms of $A$. The third formula \eqref{eq:commdiv} follows similarly: a computation in the local frame gives
\[\begin{split}
( \hdiv\,^{\mc{E}}\,\vd\,^{\mc{E}}-\vdiv\,^{\mc{E}}\hd\,^{\mc{E}})(\sum_{k=1}^n u^k e_k)= &
\sum_{k=1}^n ((\hdiv \vd\,-\vdiv \hd)u^k)e_k -\sum_{k,l=1}^n (\vdiv \vd A^k_l)u^l e_k.
\end{split}\]
Since $-\vdiv \vd a = (d-1)a$ if $a \in C^{\infty}(SM)$ corresponds to a $1$-form, the last terms in the sum become
$(d-1)A^k_l u^l e_k$. We also have $(\hdiv \vd\,-\vdiv \hd) u^k= (d-1)Xu^k$ by  \cite[Lemma 2.1]{PSU_hd} 
 and this achieves the proof of \eqref{eq:commdiv}.
\end{proof}

The next proposition states the Pestov identity with a connection. The proof is identical to the proof of \cite[Proposition 2.2]{PSU_hd} upon using the commutator formulas in Lemma \ref{lemma_basic_commutator_formulas}.

\begin{Proposition} \label{prop_pestov}
Let $(M,g)$ be a compact Riemannian manifold with or without boundary, and let $(\mc{E},\nabla^{\mc{E}})$  be a Hermitian bundle with Hermitian connection over $M$, which we pull back to $SM$. 
Then
\begin{equation*}
\norm{\vd\, ^{\mc{E}} \mathbb{X} u}^2 = \norm{\mathbb{X}\vd\, ^{\mc{E}} u}^2-
(R\,\vd\, ^{\mc{E}} u,\vd\, ^{\mc{E}} u) -(F^{\mc{E}}u,\vd\, ^\mc{E} u)+ 
(d-1)\norm{\mathbb{X} u}^2
\end{equation*}
for any $u\in C^{\infty}(SM;\mc{E})$, with $u|_{\partial(SM)} = 0$ in the boundary case. The maps $R$ and $F^{\mc{E}}$ are defined in \eqref{curvatureg} and \eqref{defFE}.
\end{Proposition}

\vspace{12pt}

\noindent 3.4. {\bf Spherical harmonics decomposition.}
We can use the spherical harmonics decomposition from Section \ref{sec_introduction} and \cite[Section 3]{PSU_hd},
$$
L^2(SM;\mc{E}) = \bigoplus_{m=0}^{\infty} H_m(SM;\mc{E}),
$$
so that any $f \in L^2(SM;\mc{E})$ has the orthogonal decomposition 
$$
f = \sum_{m=0}^{\infty} f_m.
$$
We write $\Omega_m = H_m(SM;\mc{E}) \cap C^{\infty}(SM;\mc{E})$, 
and write $\Delta^{\mc{E}} = -\vdiv\,^{\mc{E}} \vd\, ^{\mc{E}}$ for the vertical Laplacian. Notice that since $(\mc{E}, \nabla^{\mc{E}})$ are pulled back from $M$ to $SM$, we have in a local orthonormal frame $(e_1,\dots,e_n)$ the representation 
\[ \Delta^{\mc{E}} (\sum_{k=1}^nu^k e_k) = \sum_{k=1}^n(\Delta u^k) e_k \]
where $\Delta:=-\vdiv \vd$ is the vertical Laplacian on functions defined in \cite[Section 3]{PSU_hd}.
Then $\Delta^{\mc{E}} u = m(m+d-2)u$ for $u \in \Omega_m$. 
We have the following commutator formula, whose proof is identical to that of \cite[Lemma 3.6]{PSU_hd}.

\begin{Lemma} \label{lemma_xa_delta_commutator}
The following commutator formula holds:
\[[\mathbb{X},\Delta^{\mc{E}}]=2\, \vdiv\, ^{\mc{E}}\hd\, ^{\mc{E}}+(d-1)\mathbb{X}.\]
\end{Lemma}

Recall that if $a$ and $b$ are two spherical harmonics in $S^{d-1}$, where $a \in H_1(S^{d-1})$ and $b \in H_m(S^{d-1})$, then the product $ab$ is in $H_{m+1}(S^{d-1}) \oplus H_{m-1}(S^{d-1})$. Thus $\mathbb{X}$ splits as $\mathbb{X} = \mathbb{X}_+ + \mathbb{X}_-$ where 
$$
\mathbb{X}_{\pm}: \Omega_m \to \Omega_{m \pm 1}.
$$
Since the connection is Hermitian, we have $\mathbb{X}_{+}^*=-\mathbb{X}_-$.

The following special case of the Pestov identity with a connection (Proposition \ref{prop_pestov}) is very useful for studying individual Fourier coefficients of solutions of the transport equation. The proof is the same as that of \cite[Proposition 3.5]{PSU_hd}.

\begin{Proposition} \label{prop_pestov_omegam}
Let $(M,g)$ be a compact $d$-dimensional Riemannian manifold with or without boundary. If the Pestov identity with connection is applied to functions in $\Omega_m$, one obtains the identity 
\begin{equation*}
(2m+d-3)\norm{\mathbb{X}_{-}u}^{2}+\norm{\hd\, ^{\mc{E}} u}^{2}-(R\vd\,^{\mc{E}} u,\vd\, ^{\mc{E}} u)-
(F^{\mc{E}}u,\vd\, ^{\mc{E}} u)=(2m+d-1)\norm{\mathbb{X}_{+}u}^{2}
\end{equation*}
which is valid for any $u\in\Omega_m$ (with $u|_{\partial(SM)}=0$ in the boundary case).
\end{Proposition}

\vspace{12pt}

\noindent 3.5. {\bf Lower bounds.}
The following results extend \cite[Lemmas 4.3 and 4.4]{PSU_hd} to the case where a Hermitian 
connection is present. The proofs are identical, but we repeat them for completeness.

\begin{Lemma} \label{lemma_xvu_lowerbound}
If $u \in C^{\infty}(SM;\mc{E})$ and $u = \sum_{l=m}^{\infty} u_l$ with $u_l\in \Omega_l$, then 
$$
\norm{\mathbb{X} \vd\, ^{\mc{E}} u}^2 \geq \left\{ \begin{array}{ll} \frac{(m-1)(m+d-2)^2}{m+d-3} \norm{(\mathbb{X} u)_{m-1}}^2 + \frac{m(m+d-1)^2}{m+d-2} \norm{(\mathbb{X} u)_m}^2, & m \geq 2, \\[5pt]
\frac{d^{2}}{d-1}\norm{(\mathbb{X} u)_{1}}^{2}, & m = 1. \end{array} \right.
$$
If $u \in \Omega_m$, we have 
$$
\norm{\mathbb{X} \vd\, ^{\mc{E}} u}^2 \geq \left\{ \begin{array}{ll} \frac{(m-1)(m+d-2)^2}{m+d-3} \norm{\mathbb{X}_- u}^2 + 
\frac{m^2(m+d-1)}{m+1} \norm{\mathbb{X}_+ u}^2, & m \geq 2, \\[5pt] \frac{d}{2} \norm{\mathbb{X}_+ u}^2, & m = 1. \end{array} \right.
$$
\end{Lemma}

\begin{Lemma} \label{lemma_nablah_nablav_innerproduct}
If $u \in C^{\infty}(SM;\mc{E})$ and $w_l \in \Omega_l$, then 
$$
(\hd\, ^{\mc{E}}  u, \vd\, ^{\mc{E}}w_l) = ( (l+d-2) \mathbb{X}_+ u_{l-1} - l \mathbb{X}_- u_{l+1}, w_l).
$$
As a consequence, for any $u \in C^{\infty}(SM; \mc{E})$ we have the decomposition 
$$
\hd\, ^{\mc{E}}  u = \vd\, ^{\mc{E}}\left[ \sum_{l=1}^{\infty} \left( \frac{1}{l} \mathbb{X}_+ u_{l-1} - \frac{1}{l+d-2} \mathbb{X}_- u_{l+1} \right) \right] + Z(u)
$$
where $Z(u) \in C^{\infty}(SM;N\otimes \mc{E})$ satisfies $\vdiv\, ^{\mc{E}}Z(u) = 0$.
\end{Lemma}
\begin{proof}[Proof of Lemma \ref{lemma_nablah_nablav_innerproduct}]
By Lemma \ref{lemma_xa_delta_commutator}, 
\begin{align*}
 &(\hd\, ^{\mc{E}}  u, \vd\, ^{\mc{E}}w_l) = - (\vdiv\, ^{\mc{E}} \hd\, ^{\mc{E}}  u, w_l) = -\frac{1}{2} ([\mathbb{X}, \Delta^{\mc{E}}] u, w_l) + \frac{d-1}{2} (\mathbb{X} u, w_l) \\
 &= -\frac{1}{2} ([\mathbb{X}_+,  \Delta^{\mc{E}}] u + [\mathbb{X}_-,  \Delta^{\mc{E}}] u, w_l) + \frac{d-1}{2} (\mathbb{X} u, w_l) \\
 &= -\frac{1}{2} ([\mathbb{X}_+,  \Delta^{\mc{E}}] u_{l-1} + [\mathbb{X}_-,  \Delta^{\mc{E}}] u_{l+1}, w_l) + \frac{d-1}{2} (\mathbb{X}_+ u_{l-1} + \mathbb{X}_- u_{l+1}, w_l) \\
 &= \left( \frac{2l+d-3}{2} \mathbb{X}_+ u_{l-1} - \frac{2l+d-1}{2} \mathbb{X}_- u_{l+1}, w_l \right) \\
 & \qquad + \frac{d-1}{2} (\mathbb{X}_+ u_{l-1} + \mathbb{X}_- u_{l+1}, w_l)
\end{align*}
which proves the first claim. For the second one, we note that 
\[\begin{split}
(\hd\, ^{\mc{E}}  u, \vd\, ^{\mc{E}}w) = & \sum_{l=1}^{\infty} ( (l+d-2) \mathbb{X}_+ u_{l-1} - l \mathbb{X}_- u_{l+1}, w_l) \\
 = & \sum_{l=1}^{\infty} \frac{1}{l(l+d-2)} ( \vd\, ^{\mc{E}}\left[ (l+d-2) \mathbb{X}_+ u_{l-1} - l \mathbb{X}_- u_{l+1} \right], \vd\, ^{\mc{E}}w_l)
\end{split}\]
so 
$$
(\hd\, ^{\mc{E}}  u - \vd\, ^{\mc{E}}\left[ \sum_{l=1}^{\infty} \left( \frac{1}{l} \mathbb{X}_+ u_{l-1} - \frac{1}{l+d-2} \mathbb{X}_- u_{l+1} \right) \right], \vd\, ^{\mc{E}}w) = 0
$$
for all $w \in C^{\infty}(SM;\mc{E})$.
\end{proof}

\begin{proof}[Proof of Lemma \ref{lemma_xvu_lowerbound}]
Let $u = \sum_{l=m}^{\infty} u_l$ with $m \geq 2$. First note that 
$$
\norm{\mathbb{X} \vd\, ^{\mc{E}}u}^2 = \norm{\vd\, ^{\mc{E}}\mathbb{X} u - \hd\, ^{\mc{E}}  u}^2.
$$
We use the decomposition in Lemma \ref{lemma_nablah_nablav_innerproduct}, which implies that 
\begin{multline*}
\vd\, ^{\mc{E}}\mathbb{X} u - \hd\, ^{\mc{E}}  u = \\
 \vd\, ^{\mc{E}}\Bigg[ \left( 1 + \frac{1}{m+d-3} \right) (\mathbb{X} u)_{m-1} + \left( 1 + \frac{1}{m+d-2} \right) (\mathbb{X} u)_m  + \sum_{l=m+1}^{\infty} w_l \Bigg] + Z
\end{multline*}
where $w_l \in \Omega_l$ for $l \geq m+1$ are given by 
$$
w_l = (\mathbb{X} u)_l - \frac{1}{l} \mathbb{X}_+ u_{l-1} + \frac{1}{l+d-2} \mathbb{X}_- u_{l+1}
$$
and where $Z \in C^{\infty}(SM;N\otimes \mc{E})$ satisfies $\vdiv\, ^{\mc{E}}Z = 0$. Taking the $L^2$ norm squared, and noting that the term $\vd\,^{\mc{E}}(\,\cdot\,)$ is orthogonal to the $\vdiv\,^{\mc{E}}$-free vector field $Z$, gives 
\[\begin{split}
\norm{\mathbb{X} \vd\, ^{\mc{E}}u}^2 = & \frac{(m-1)(m+d-2)^2}{m+d-3} \norm{(\mathbb{X} u)_{m-1}}^2 + \frac{m(m+d-1)^2}{m+d-2} \norm{(\mathbb{X} u)_m}^2 \\
& + \sum_{l=m+1}^{\infty} \norm{\vd\, ^{\mc{E}}w_l}^2 + \norm{Z}^2.  
\end{split}\]
The claims for $m=1$ or for $u \in \Omega_m$ are essentially the same.
\end{proof}

\vspace{12pt}

\noindent 3.6. {\bf Identification with trace free symmetric tensors and conformal invariance.}\label{Identifi} 
For $\mc{E}$ being the trivial complex line bundle, there is an identification of $\Omega_m$ with the smooth {\it trace free} symmetric tensor fields of degree $m$ on $M$ which we denote by $\Theta_m$ \cite{DS,GK2}.
More precisely, as in \cite{DS} we start with $\lambda:C^{\infty}(M; \otimes^m_ST^*M)\to C^{\infty}(SM)$ 
being the map which takes a symmetric  $m$-tensor $f$ and maps it into the section 
$SM\ni (x,v)\mapsto f_{x}(v,\dots,v)$.
The map $\lambda$ turns out to be an isomorphism between $\Omega_m$
and $\Theta_m$.
In fact up to a factor which depends on $m$ and $d$  only, it is a linear isometry when the spaces are endowed with the obvious $L^2$-inner products; this is detailed in \cite[Lemma 2.4]{DS} and \cite[Lemma 2.9]{GK2}.
There is a natural operator $D:=\mc{S}\circ \nabla: C^\infty(M; \otimes^m_ST^*M)\to  
C^\infty(M; \otimes^{m+1}_ST^*M)$ where $\nabla$ is the Levi-Civita connection acting on tensors and 
$\mc{S}:\otimes^mT^*M\to \otimes_S^mT^*M$ is the orthogonal projection from the space of $m$-tensors 
to symmetric ones. The trace map $\mc{T}: \otimes^m T^*M\to \otimes^{m-2}T^*M$ is defined by 
\[ \mc{T}u (v_1,\dots,v_{m-2})=\sum_{j=1}^d u(E_j,E_j, v_1,\dots,v_{m-2}), \quad v_i\in TM\]
where $(E_j)_{j=1,\dots,d}$ is an orthonormal basis of $TM$ for $g$. Then the adjoint $D^*=-\mc{T}\circ \nabla$ 
is (minus) the divergence operator. Then in \cite[Section 10]{DS}
we find the formula
\[X_{-}u=-\frac{m}{d+2m-2}\lambda D^*\lambda^{-1}u,\]
for $u\in\Omega_m$.
The expression for $X_{+}$ in terms of tensors is as follows. If $\mc{P}$ denotes orthogonal projection onto 
$\Theta_{m+1}$ then
\[X_{+}u=\lambda \mc{P}D\lambda^{-1}u\]
for $u\in\Omega_m$.
In other words, up to $\lambda$, $X_{+}$ is $\mc{P}D$ and $X_{-}$ is 
$-\frac{m}{d+2m-2}D^*$. The operator $X_+$, at least for $m=1$, has many names and is known as the conformal Killing operator, trace-free deformation tensor, or Ahlfors operator.

Under this identification $\mbox{\rm Ker}\,X_{+}$ consists of the {\it conformal Killing symmetric tensor fields}, a finite dimensional space. It is well known that the dimension of this space depends only on the conformal class
of the metric, but let us look at this in more detail.  Consider a new metric of the form $\tilde{g}=e^{2\varphi}\,g$.
The first observation is that the space $\Theta_{m}$ is the same for both metrics and thus the operator $\mc{P}$ 
is also the same for both metrics. To see how $D$ changes under conformal change, we see from Koszul formula 
that the two Levi-Civita connections $\nabla$  and $\tilde{\nabla}$ associated to $g$ and $\tilde{g}$, acting 
on $1$-forms $T\in C^\infty(M;T^*M)$, are related by
\begin{equation} \label{nablatilde_nabla_relation}
\tilde{\nabla}T= \nabla T-2\mc{S}(d\varphi\otimes T)+T(\tilde{\nabla} \varphi) \tilde{g}.
\end{equation}
Let $\tilde{D}=\mc{S}\tilde{\nabla}$, then since $\mc{P}$ corresponds to orthogonal projection to the space of trace-free tensors, the parts involving $\tilde{g}$ will disappear in the computation below when applying $\mc{P}$. We then get for 
$T=\sum_{\sigma\in \Pi_m}T_{\sigma(1)}\otimes \dots \otimes T_{\sigma(m)} 
\in C^\infty(M; \otimes_S^mT^*M)$ with $\Pi_m$ the set of permutations of $(1,\dots,m)$ 
and $T_j\in C^\infty(M;T^*M)$ that 
\[
e^{-2m\varphi}\mc{P}\tilde{D}(e^{2m\varphi}T)= \mc{P}\mc{S}\tilde{\nabla}T+2m\mc{P}\mc{S}(d\varphi\otimes T).
\] 
Since  
\[\tilde{\nabla}_Y(T_{\sigma(1)}\otimes \dots \otimes T_{\sigma(m)}) = (\tilde{\nabla}_Y T_{\sigma(1)}) \otimes \dots \otimes T_{\sigma(m)} +\dots+ T_{\sigma(1)} \otimes \dots \otimes (\tilde{\nabla}_Y T_{\sigma(m)}),\]
the formula \eqref{nablatilde_nabla_relation} and symmetrization give 
\begin{equation*}
\mc{P}\mc{S}\tilde{\nabla}T = \mc{P}\mc{S}(\nabla T- 2m d\varphi \otimes T).
\end{equation*}
Then we deduce that 
\begin{equation}\label{relationconforme} 
e^{-2m\varphi}\mc{P}\tilde{D}(e^{2m\varphi}T) = \mc{P}DT.
\end{equation}

If now we have a general Hermitian bundle $\mc{E}$ with connection $\nabla^{\mc{E}}$, we can proceed similarly.
The map $\la$ extends naturally to 
$\la: C^{\infty}(M; \otimes^m_ST^*M\otimes \mc{E})\to C^{\infty}(SM; \mc{E})$ and is an isomorphism between 
$\Theta_m$ and $\Omega_m$ where now $\Theta_m$ is the space of trace-free sections in 
$C^{\infty}(M; \otimes^m_ST^*M\otimes \mc{E})$ and $\Omega_m = H_m(SM;\mc{E}) \cap C^{\infty}(SM;\mc{E})$. 
We can define $D^\mc{E}$ acting on $C^\infty(M; \otimes_S^m T^*M\otimes \mc{E})$ by 
\[ D^{\mc{E}}u:= \mc{S}\nabla^\mc{E}(u)\]
where $\mc{S}$ means the symmetrization $\mc{S}:T^*M\otimes (\otimes_S^m T^*M)\otimes \mc{E} \to 
(\otimes_S^{m+1} T^*M)\otimes \mc{E}$.
Using a local orthonormal frame $(e_1,\dots,e_n)$ the connection $\nabla^{\mc{E}}=d+A$ for some connection $1$-form with values in skew-Hermitian matrices, and in this frame 
\[ D^{\mc{E}} (\sum_{k=1}^n u^k \otimes e_k) = \sum_{k,l=1}^n (Du^k)\otimes e_k+ \mc{S}(A^k_l \otimes u^l)\otimes e_k.\]
Then we get 
\[\mathbb{X}_{+}\lambda u=\lambda \mc{P}D^\mc{E}u\]
and hence the elements in the kernel of $\mathbb{X}_{+}$ are in 1-1 correspondence with tensors 
$u\in\Theta_m$ with $\mc{P}D^{\mc{E}}u=0$. Since in the local frame $(e_1,\ldots, e_n)$ we have, using 
\eqref{relationconforme}, that for $u=\sum_{k=1}^n u^k \otimes e_k$
\[\mc{P}\tilde{D}^{\mc{E}}(e^{2m\varphi}u)=\sum_{k=1}^n
e^{2m\varphi}\mc{P} (Du^k) \otimes e_k +\sum_{k,l=1}^n
\mc{P}\mc{S}(A^k_l \otimes e^{2m\varphi}u^l)\otimes e_k=e^{2m\varphi}\mc{P}D^{\mc{E}}u\]
we see that the dimension of the space of twisted conformal Killing tensors is also a conformal invariant.

\section{Finite degree} \label{sec_finite_degree}

In this section we will prove the finite degree part of Theorem \ref{main_thm_closed_finitedegree} in the closed case, as well as its analogue in the boundary case. In Section \ref{sec_twisted_ckts_raytransform} we will consider the corresponding improved results (stating that $u$ has degree one smaller than $f$) in those cases where twisted CKTs do not exist. The underlying idea of the proof of finite degree is that for sufficiently high enough Fourier modes, the sectional curvature overtakes the contribution of the connection and the Higgs field in the Pestov identity. This idea first appeared in \cite{Pa2} in 2D and its implementation in higher dimensions is one the contributions of the present paper.

We use the notations of Section \ref{sec_pestov}. For simplicity we first discuss the case where no Higgs field is present, and prove the following result:

\begin{Theorem}\label{thm:finitedegree} 
Let $(M,g)$ be a compact manifold of negative sectional curvature with or without boundary 
and let $\mc{E}$ be a Hermitian bundle equipped with a Hermitian connection $\nabla^\mc{E}$. 
Suppose $u\in C^{\infty}(SM;\mc{E})$ solves
\[\mathbb{X}u=f\]
where $f$ has finite degree (and $u|_{\partial(SM)}=0$ in the boundary case).  Then $u$ has finite degree.
\end{Theorem}

The proofs in the closed case and in the boundary case are identical, and we will henceforth consider only closed manifolds in this section. We give two proofs of Theorem \ref{thm:finitedegree}. The first proof is based on applying the Pestov identity with a connection to the tail of a Fourier series, which gives the following result. We use the notation 
\[
T_{\geq m} u = \sum_{k=m}^{\infty} u_k, \qquad u \in L^2(SM;\mc{E}).
\]

\begin{Lemma} \label{lemma_pestov_negative1}
Let $(M,g)$ be a closed manifold such that the sectional curvatures are uniformly 
bounded above by $-\kappa$ for some $\kappa > 0$. 
Let $(\mc{E},\nabla^{\mc{E}})$ be a Hermitian bundle with Hermitian connection over $M$, 
and assume that $m \geq 1$ is so large that 
\[
\lambda_m \geq \frac{2 \norm{F^{\mc{E}}}_{L^{\infty}}^2}{\kappa^2}
\]
where $\lambda_m = m(m+d-2)$, $F^{\mc{E}}$ is the curvature operator of $\nabla^{\mc{E}}$ defined 
by \eqref{defFE}, and 
\[
\norm{F^{\mc{E}}}_{L^{\infty}} = \norm{F^{\mc{E}}}_{L^{\infty}(SM ; N \otimes \mathrm{End}(\mc{E}))}.
\]
Then we have the inequality 
\[
\frac{\kappa}{4} \norm{\vd\, ^{\mc{E}}T_{\geq m} u}^2 \leq \norm{\vd\, ^{\mc{E}}T_{\geq m+1} \mathbb{X} u}^2, \qquad u \in C^{\infty}(SM; \mc{E}).
\]
\end{Lemma}
\begin{proof}
We will do the proof for $m \geq 2$ (the argument for $m=1$ is similar). Since $T_{\geq m+1} \mathbb{X} T_{\geq m} u = T_{\geq m+1} \mathbb{X} u$, it is enough to prove that 
$$
\frac{\kappa}{4} \norm{\vd\, ^{\mc{E}}u}^2 \leq \norm{\vd\, ^{\mc{E}}T_{\geq m+1} \mathbb{X} u}^2, \qquad u \in T_{\geq m} C^{\infty}(SM;\mc{E}).
$$
If $u \in T_{\geq m} C^{\infty}(SM;\mc{E})$, the Pestov identity yields 
\begin{align*}
 &\norm{\vd\, ^{\mc{E}}T_{\geq m+1} \mathbb{X} u}^2 + m(m+d-2) \norm{(\mathbb{X} u)_m}^2 + (m-1)(m+d-3) \norm{(\mathbb{X} u)_{m-1}}^2 \\
 &= \norm{\vd\, ^{\mc{E}}\mathbb{X} u}^2 \\
 &= \norm{\mathbb{X}\vd\, ^{\mc{E}}u}^2-(R\,\vd\, ^{\mc{E}}u,\vd\, ^{\mc{E}}u) -(F^{\mc{E}}u,\vd\, ^{\mc{E}}u)+ 
 (d-1)\norm{\mathbb{X} u}^2 \\
\intertext{which is, using Lemma \ref{lemma_xvu_lowerbound} and the fact that the sectional curvatures are $\leq -\kappa$,}
 &\geq \frac{(m-1)(m+d-2)^2}{m+d-3} \norm{(\mathbb{X} u)_{m-1}}^2 + \frac{m(m+d-1)^2}{m+d-2} \norm{(\mathbb{X} u)_m}^2 \\
 &\qquad + \kappa \norm{\vd\, ^{\mc{E}}u}^2 -(F^{\mc{E}}u,\vd\, ^{\mc{E}}u)+ (d-1)\norm{\mathbb{X} u}^2.
\end{align*}
We obtain in particular that 
$$
\kappa \norm{\vd\, ^{\mc{E}}u}^2 \leq \norm{\vd\, ^{\mc{E}}T_{\geq m+1} \mathbb{X} u}^2 + (F^{\mc{E}}u,\vd\, ^{\mc{E}}u).
$$
Using the inequality $\abs{(F^{\mc{E}}u,\vd\, ^{\mc{E}}u)} \leq \frac{1}{2}(\frac{1}{\kappa} \norm{F^{\mc{E}}}_{L^{\infty}}^2 \norm{u}^2 + \kappa \norm{\vd\, ^{\mc{E}}u}^2)$, we get 
$$
\frac{\kappa}{2} \norm{\vd\, ^{\mc{E}}u}^2 \leq \norm{\vd\, ^{\mc{E}}T_{\geq m+1} \mathbb{X} u}^2 + \frac{\norm{F^{\mc{E}}}_{L^{\infty}}^2}{2 \kappa} \norm{u}^2.
$$
Finally, we have $\frac{\norm{F^{\mc{E}}}_{L^{\infty}}^2}{2 \kappa} \norm{u}^2 \leq \frac{\norm{F^{\mc{E}}}_{L^{\infty}}^2}{2 \kappa \lambda_m} \norm{\vd\, ^{\mc{E}}u}^2$, and thus if $m$ is so large that 
$$
\lambda_m \geq \frac{2 \norm{F^{\mc{E}}}_{L^{\infty}}^2}{\kappa^2},
$$
then we have 
\begin{equation*}
\frac{\kappa}{4} \norm{\vd\, ^{\mc{E}}u}^2 \leq \norm{\vd\, ^{\mc{E}}T_{\geq m+1} \mathbb{X} u}^2.   \qedhere
\end{equation*}
\end{proof}

For possible later purposes, we record another lemma which follows easily from the previous one and states that if $\mathbb{X} u$ is smooth in the vertical variable, then so is $u$. (This lemma will not be used anywhere in this paper.)

\begin{Lemma}
Let $(M,g)$ and $(\mc{E},\nabla^{\mc{E}})$ as in Lemma \ref{lemma_pestov_negative1}
and assume that $m \geq 1$ is so large that 
\[
\lambda_m \geq \frac{2 \norm{F^{\mc{E}}}_{L^{\infty}}^2}{\kappa^2}.
\]
Let also $\eps > 0$. There is $C = C(\kappa, \eps) > 0$ so that for any $N \geq 1$ we have 
\[
\sum_{k=m+1}^{\infty} \lambda_{k}^{N-1/2-\eps} \norm{u_{k}}^2 \leq \frac{C}{m^{2\eps}} \sum_{l=m+1}^{\infty} \lambda_l^{N} \norm{(\mathbb{X} u)_l}^2, \qquad u \in C^{\infty}(SM;\mc{E}).
\]
\end{Lemma}
\begin{proof}
By Lemma \ref{lemma_pestov_negative1}, 
\begin{align*}
\frac{\kappa}{4} \norm{\vd\, ^{\mc{E}}T_{\geq m} u}^2 &\leq \sum_{l=m+1}^{\infty} \lambda_l \norm{(\mathbb{X} u)_l}^2 = \sum_{l=m+1}^{\infty} \lambda_l^{1-N} \lambda_l^{N} \norm{(\mathbb{X} u)_l}^2 \\
 &\leq \lambda_{m+1}^{1-N} \sum_{l=m+1}^{\infty} \lambda_l^{N} \norm{(\mathbb{X} u)_l}^2.
\end{align*}
Thus in particular $\frac{\kappa}{4} \lambda_{m+1}^N \norm{u_{m+1}}^2 \leq \sum_{l=m+1}^{\infty} \lambda_l^{N} \norm{(\mathbb{X} u)_l}^2$. This shows that 
\[
\sum_{k=m+1}^{\infty} \lambda_{k}^{N-1/2-\eps} \norm{u_{k}}^2 \leq \frac{C}{m^{2\eps}} \sup_{k \geq m+1} \lambda_{k}^N \norm{u_{k}}^2 \leq \frac{C}{m^{2\eps}} \sum_{l=m+1}^{\infty} \lambda_l^{N} \norm{(\mathbb{X} u)_l}^2. \qedhere
\]
\end{proof}

\begin{proof}[First proof of Theorem \ref{thm:finitedegree}]
If $(M,g)$ has sectional curvatures bounded above by $-\kappa$ where $\kappa > 0$, and if $f$ has degree $l$, we choose $m \geq 1$ so large that $\lambda_m \geq \frac{2 \norm{F^{\mc{E}}}_{L^{\infty}}^2}{\kappa^2}$ and also $m \geq l$. Then $T_{\geq m+1} \mathbb{X} u = 0$, thus by Lemma \ref{lemma_pestov_negative1} $T_{\geq m} u = 0$ so $u$ has degree less or equal to $m-1$.
\end{proof}

To deal with the case of nonzero Higgs field, it is convenient to use another proof of Theorem \ref{thm:finitedegree}. We first give the argument for $\Phi=0$. If $(M,g)$ has negative curvature and $d \neq 4$, the next result implies in particular that 
\[
\norm{\mathbb{X}_{-}u}\leq \norm{\mathbb{X}_{+}u}, \qquad u \in \Omega_m, \ \text{$m$ sufficiently large}.
\]
This is an analogue of the Beurling contraction property that was discussed in \cite{PSU_hd} 
in the case of the trivial line bundle $\mc{E}=M \times \C$ with flat connection.

\begin{Lemma} \label{lemma_beurling_contraction}
Let $(M,g)$ and $(\mc{E},\nabla^{\mc{E}})$ as in Lemma \ref{lemma_pestov_negative1}
and assume that $m \geq 1$ is so large that 
\[
\lambda_m \geq \frac{4 \norm{F^{\mc{E}}}_{L^{\infty}}^2}{\kappa^2}
\]
where $\lambda_m = m(m+d-2)$. Then for any $u \in \Omega_m$ we have 
\[
\norm{\mathbb{X}_{-}u}^{2} + c_m \norm{u}^2 \leq d_m \norm{\mathbb{X}_{+}u}^{2}
\]
where $c_m$ and $d_m$ can be chosen as  
\[
c_m = \frac{\kappa m}{4}, \qquad d_m = \left\{ \begin{array}{cl} 1, & d \neq 3 \text{ and } m \geq 2, \\ \frac{d+2}{2d-2}, & d \neq 3 \text{ and } m = 1, \\ 1+\frac{1}{(m+1)^2(2m-1)}, & d = 3. \end{array} \right.
\]
\end{Lemma}
\begin{proof}
Let $u \in \Omega_m$. From Proposition \ref{prop_pestov_omegam} we have the identity 
\begin{equation*}
(2m+d-3)\norm{\mathbb{X}_{-}u}^{2}+\norm{\hd\,^{\mc{E}} u}^{2}-(R\vd\, ^{\mc{E}}u,\vd\, ^{\mc{E}}u)-(F^{\mc{E}}u,\vd\, ^{\mc{E}}u)=(2m+d-1)\norm{\mathbb{X}_{+}u}^{2}.
\end{equation*}
By Lemma \ref{lemma_nablah_nablav_innerproduct} 
\begin{align*}
\norm{\hd\,^{\mc{E}} u}^{2} &= \norm{\frac{1}{m+1} \vd\, ^{\mc{E}}\mathbb{X}_+ u - \frac{1}{m+d-3} \vd\, ^{\mc{E}}\mathbb{X}_- u + Z(u)}^2 \\
 &\geq \frac{m+d-1}{m+1} \norm{\mathbb{X}_+ u}^2 + \frac{m-1}{m+d-3} \norm{\mathbb{X}_- u}^2
\end{align*}
using that $Z(u)$ is $L^2$-orthogonal to the $\vd\, ^{\mc{E}}(\,\cdot\,)$ terms because $\vdiv\,^{\mc{E}}Z(u) = 0$. Thus 
\begin{multline*}
\left[ 2m+d-3 + \frac{m-1}{m+d-3} \right] \norm{\mathbb{X}_{-}u}^{2}-(R\vd\, ^{\mc{E}}u,\vd\, ^{\mc{E}}u)-(F^{\mc{E}}u,\vd\, ^{\mc{E}}u) \\
 \leq \left[ 2m+d-1- \frac{m+d-1}{m+1} \right] \norm{\mathbb{X}_{+}u}^{2}.
\end{multline*}
The issue is to show that for large $m$, the term involving $R$ wins over the term involving $F^{\mc{E}}$. Indeed, the assumption on sectional curvature yields 
\[
-(R\vd\, ^{\mc{E}}u,\vd\, ^{\mc{E}}u) \geq \kappa \norm{\vd\, ^{\mc{E}}u}^{2}=\kappa \lambda_m \norm{u}^{2}.
\]
On the other hand, since $M$ is compact we have 
\[
|(F^{\mc{E}} u,\vd\, ^{\mc{E}}u)|\leq \norm{F^{\mc{E}}}_{L^{\infty}} \norm{u}\norm{\vd\, ^{\mc{E}}u}= \norm{F^{\mc{E}}}_{L^{\infty}} \lambda_m^{1/2}\norm{u}^{2}.
\]
The assumption on $m$ implies that $\frac{\kappa}{2} \lambda_m \geq \norm{F^{\mc{E}}}_{L^{\infty}} \lambda_m^{1/2}$, which gives 
\[
-(R\vd\, ^{\mc{E}}u,\vd\, ^{\mc{E}}u)-(F^{\mc{E}}u,\vd\, ^{\mc{E}}u) \geq \frac{\kappa}{2} \lambda_m \norm{u}^2.
\]
Putting these facts together implies that 
\[
\norm{\mathbb{X}_{-}u}^{2} + c_m \norm{u}^2 \leq d_m \norm{\mathbb{X}_{+}u}^{2}
\]
where $c_m$ and $d_m$ may be chosen as stated.
\end{proof}

\begin{proof}[Second proof of Theorem \ref{thm:finitedegree}]
Let $\mathbb{X}u = f$ where $f$ has degree $l$. Looking at Fourier coefficients we have $(\mathbb{X} u)_k = 0$ for $k \geq l+1$, meaning that 
$$
\mathbb{X}_+ u_k = -\mathbb{X}_- u_{k+2}, \qquad k \geq l.
$$
Let $m \geq l$ and let also $m$ satisfy the condition in Lemma \ref{lemma_beurling_contraction}. Using Lemma \ref{lemma_beurling_contraction} and the identity above repeatedly, we obtain for any $N \geq 0$ that 
\begin{align*}
 &\norm{\mathbb{X}_- u_m}^2 + c_m \norm{u_m}^2 \leq d_m \norm{\mathbb{X}_+ u_m}^2 = d_m \norm{\mathbb{X}_- u_{m+2}}^2 \\
 &\leq d_m d_{m+2} \norm{\mathbb{X}_+ u_{m+2}}^2 = d_m d_{m+2} \norm{\mathbb{X}_- u_{m+4}}^2 \leq \ldots \leq \left[ \prod_{j=0}^N d_{m+2j} \right] \norm{\mathbb{X}_- u_{m+2N+2}}.
\end{align*}
Since $\mathbb{X}_- u \in L^2$, we have $\norm{\mathbb{X}_- u_k} \to 0$ as $k \to \infty$. Also, the constant $\prod_{j=0}^N d_{m+2j}$ stays finite as $N \to \infty$. This shows that $u_m = 0$ for $m$ sufficiently large.
\end{proof}

Another immediate consequence of Lemma \ref{lemma_beurling_contraction} is the following theorem, which implies Theorem \ref{thm_main_ckt_hd} when combined with the conformal invariance discussed at the end of Section \ref{sec_pestov}.

\begin{Theorem}
Let $(M,g)$ be a closed manifold satisfying $K \leq -\kappa$ for some $\kappa > 0$. Let $\mc{E}$ be a Hermitian bundle with Hermitian connection $\nabla$, and assume that $m \geq 1$ satisfies 
\[
m(m+d-2) \geq \frac{4 \norm{F^{\mc{E}}}_{L^{\infty}}^2}{\kappa^2}.
\]
Then any $u \in \Omega_m$ satisfying $\mathbb{X}_+ u = 0$ must be identically zero.
\end{Theorem}

In the rest of this section, we explain how to include a Higgs field in Theorem \ref{thm:finitedegree}:

\begin{Theorem} \label{thm:finitedegree_higgs}
Let $(M,g)$ be a compact manifold with negative sectional curvature, with or without boundary, let 
$(\mc{E},\nabla^{\mc{E}})$ be a Hermitian bundle over $M$ with Hermitian connection, 
and let $\Phi$ be a skew-Hermitian Higgs field. Suppose $u\in C^{\infty}(SM;\mc{E})$ 
(with $u|_{\partial(SM)}=0$ in the boundary case) solves
\[(\mathbb{X}+\Phi) u=f\]
where $f$ has finite degree. Then $u$ has finite degree.
\end{Theorem}

We will follow the strategy in \cite{P1} which considered the case where $\dim(M)=2$. Again we will only do the proof for closed manifolds (the boundary case is identical as long as we insist that $u|_{\partial(SM)}=0$).

\begin{proof}[Proof of Theorem \ref{thm:finitedegree_higgs}]
We first assume that $d \neq 3$ (the case $d=3$ is a little different). By Lemma \ref{lemma_beurling_contraction}, since the sectional curvature of $M$ is negative, there exist constants $c_m>0$ with $c_{m}\to\infty$ as $m \to \infty$�and a positive integer $l$ such that
\begin{equation}
\norm{\mathbb{X}_{+}u_m}^2\geq \norm{\mathbb{X}_{-}u_{m}}^2+c_{m}\norm{u_{m}}^2
\label{eq:basicineq}
\end{equation}
for all $m\geq l$ and $u_m\in\Omega_m$. Write $u=\sum u_m$. We know that for all $m$ sufficiently large
\begin{equation}
\mathbb{X}_{+}u_{m-1}+\mathbb{X}_{-}u_{m+1}+\Phi u_m=0.
\label{eq:recurrence}
\end{equation}
Combining (\ref{eq:basicineq}) and (\ref{eq:recurrence}) we derive
\begin{equation}
\norm{\mathbb{X}_{+}u_{m+1}}^2\geq \norm{\mathbb{X}_{+}u_{m-1}}^2+c_{m+1}\norm{u_{m+1}}^2+\norm{\Phi u_m}^{2}
+2\mathrm{Re}(\mathbb{X}_{+}u_{m-1},\Phi u_m).
\label{eq:combined}
\end{equation}
The rest of the proof hinges on controlling the term $\mathrm{Re}(\mathbb{X}_{+}u_{m-1},\Phi u_m)$.

Given an element $\alpha\in\Omega_1$ we write $i_{\alpha}u:=\alpha u$. Multiplication by an element of degree one has the following property: if $u\in\Omega_m$, then $i_{\alpha}u\in \Omega_{m-1}\oplus\Omega_{m+1}$ and hence we may write $i_{\alpha}u=i^{-}_{\alpha}u+i^{+}_{\alpha}u$
where $i^{\pm}_{\alpha}u\in\Omega_{m\pm1}$. For a smooth section $U$ of the bundle $\mc{F}:={\rm End}(\mc{E})$ 
we write $\mathbb{X}U\in C^{\infty}(M; \mc{F})$ for the element $(\mathbb{X}U)f:=\mathbb{X}(Uf)-U(\mathbb{X}f)$ if 
$f\in C^{\infty}(M;\mc{E})$ is any section (this corresponds to $\nabla^{\mc{F}}_X U$ where $\nabla^{\mc{F}}$ is the natural connection induced by $\nabla^{\mc{E}}$ on $\mc{F}$).
In a local trivialization where we write $\nabla^{\mc{E}}=d+A$ for some connection $1$-form, 
one has $\mathbb{X}U=XU+[A,U]$.

We now prove an auxiliary lemma:

\begin{Lemma} \label{lemma:aux}
The following identity holds for $\Phi$ skew-Hermitian:
\[(\mathbb{X}_{+}u_{m-1},\Phi u_{m})+\overline{(\mathbb{X}_{+}u_{m-2},\Phi u_{m-1})}
=-(u_{m-1},i^{-}_{\mathbb{X}\Phi}u_m)-\norm{\Phi u_{m-1}}^{2}.\]
\end{Lemma}
\begin{proof} We observe first that
\[\mathbb{X}(\Phi u_m)=(\mathbb{X}\Phi) u_{m}+\Phi \mathbb{X}u_m=i_{\mathbb{X}\Phi}^{-}u_{m}+\Phi \mathbb{X}_{-}u_{m}
+i_{\mathbb{X}\Phi}^{+}u_{m}+\Phi \mathbb{X}_{+}u_{m}\]
and thus
\[\mathbb{X}_{-}(\Phi u_m)=i_{\mathbb{X}\Phi}^{-}u_{m}+\Phi \mathbb{X}_{-}u_{m}.\]
Now compute using the above, (\ref{eq:recurrence}) and $\Phi$ skew-Hermitian:
\begin{align*}
(\mathbb{X}_{+}u_{m-1},\Phi u_{m})&=-(u_{m-1},\mathbb{X}_{-}(\Phi u_m))\\
&=-(u_{m-1},i^{-}_{\mathbb{X}\Phi}u_{m})-(u_{m-1},\Phi \mathbb{X}_{-}u_{m})\\
&=-(u_{m-1},i^{-}_{\mathbb{X}\Phi}u_{m})+(u_{m-1},\Phi(\Phi u_{m-1}+\mathbb{X}_{+}u_{m-2}))\\
&=-(u_{m-1},i^{-}_{\mathbb{X}\Phi}u_{m})-\norm{\Phi u_{m-1}}^{2}+(u_{m-1},\Phi \mathbb{X}_{+}u_{m-2})\\
&=-(u_{m-1},i^{-}_{\mathbb{X}\Phi}u_{m})-\norm{\Phi u_{m-1}}^{2}-(\Phi u_{m-1},\mathbb{X}_{+}u_{m-2})\\
\end{align*}
and the lemma is proved.
\end{proof}

The lemma suggests to consider (\ref{eq:combined}) for $m$ and $m-1$. Adding them we derive:

\begin{align*}
\norm{\mathbb{X}_{+}u_{m+1}}^2+\norm{\mathbb{X}_{+}u_{m}}^2&\geq\norm{\mathbb{X}_{+}u_{m-1}}^2+c_{m+1}\norm{u_{m+1}}^2+\norm{\Phi u_m}^{2}\\
&+\norm{\mathbb{X}_{+}u_{m-2}}^2+c_{m}\norm{u_{m}}^2+\norm{\Phi u_{m-1}}^{2}\\
&+2\mathrm{Re}(\mathbb{X}_{+}u_{m-1},\Phi u_m)+2\mathrm{Re}(\mathbb{X}_{+}u_{m-2},\Phi u_{m-1}).
\end{align*}
If we set $a_{m}:=\norm{\mathbb{X}_{+}u_{m}}^2+\norm{\mathbb{X}_{+}u_{m-1}}^2$ and we use Lemma \ref{lemma:aux}
we obtain
\begin{align*}
a_{m+1}&\geq a_{m-1}+c_{m+1}\norm{u_{m+1}}^2+c_{m}\norm{u_{m}}^2+\norm{\Phi u_{m}}^{2}-\norm{\Phi u_{m-1}}^{2}-2\mathrm{Re}(u_{m-1},i_{\mathbb{X}\Phi}^{-}u_{m})\\
&\geq a_{m-1}+c_{m+1}\norm{u_{m+1}}^2+c_{m}\norm{u_{m}}^2+\norm{\Phi u_{m}}^{2}-\norm{\Phi u_{m-1}}^{2}-\norm{u_{m-1}}^{2}-\norm{i_{\mathbb{X} \Phi}^{-}u_{m}}^{2}.
\end{align*}

Since $M$ is compact there exist positive constants $B$ and $C$ such that
\begin{align*}
&\norm{\Phi f}^{2}\leq (B-1)\norm{f}^{2}\\
&\norm{i^{-}_{\mathbb{X}\Phi}f}^{2}\leq C\norm{f}^{2}
\end{align*}
for any $f\in C^{\infty}(SM;\mc{E})$. Therefore
\[a_{m+1}\geq a_{m-1}+r_{m}\]
where
\[r_{m}:=-B\norm{u_{m-1}}^{2}+c_{m+1}\norm{u_{m+1}}^{2}+(c_{m}-C)\norm{u_{m}}^{2}.\]
Now choose a positive integer $N_{0}$ large enough so that for $m\geq N_{0}$ equations (\ref{eq:basicineq}) and (\ref{eq:recurrence}) hold and we have
\[c_{m}>\max\{B,C\}.\]
Let $m=N+1+2k$, where $k$ is a non-negative integer and $N$ is an integer with $N\geq N_0$. Note that from the definition
of $r_m$ and our choice of $N$ we have
\[r_{m}+r_{m-2}+\cdots+r_{N+1}\geq -B\norm{u_{N}}^2.\]
Thus
\[a_{m+1}\geq a_{N}+r_{m}+r_{m-2}+\cdots+r_{N+1}\geq a_{N}-B\norm{u_{N}}^2.\]
From the definition of $a_{m}$ and (\ref{eq:basicineq}) we know that
$a_{N}\geq c_{N}\norm{u_{N}}^{2}$ and hence
\[a_{m+1}\geq (c_{N}-B)\norm{u_{N}}^{2}.\]
Since the function $u$ is smooth, also $\mathbb{X}_+ u$ is smooth and $\mathbb{X}_{+}(u_m) = (\mathbb{X}_{+} u)_m$ must tend to zero
in the $L^{2}$-topology as $m\to\infty$. Hence $a_{m+1}\to 0$
as $k\to\infty$ which in turns implies that $u_{N}=0$ for any $N\geq N_0$, thus concluding that
$u$ has finite degree as desired.

We briefly indicate the modifications for $\dim(M)=3$. Inequality (\ref{eq:basicineq}) changes to
\begin{equation}
d_m\norm{\mathbb{X}_{+}u_m}^2\geq \norm{\mathbb{X}_{-}u_{m}}^2+c_{m}\norm{u_{m}}^2
\label{eq:basicineq3}
\end{equation}
where
\[d_m=1+\frac{1}{(m+1)^2(2m-1)}.\]
With the same definitions of $a_m$ and $r_m$ as above one arrives at the inequality
\[d_{m}a_{m+1}\geq a_{m-1}+r_m.\]
With this inequality one derives ($d_m\geq 1$ for all $m$):
\[ \left(\prod_{j=0}^{k}d_{m-2j}\right)a_{m+1}\geq a_{N}+r_{m}  +\cdots+r_{N+1}\geq a_{N}-B\norm{u_{N}}^2.\]
From the definition of $a_{m}$ and (\ref{eq:basicineq3}) we know that
$d_{N}a_{N}\geq c_{N}\norm{u_{N}}^{2}$ and hence
\[ \left(\prod_{j=0}^{k}d_{m-2j}\right)a_{m+1}\geq (\frac{c_{N}}{d_{N}}-B)\norm{u_{N}}^{2}.\]
Now we need to choose $N_0$ such that $\frac{c_{N_0}}{d_{N_0}}-B>0$. This is possible since $c_m\to\infty$ and
$d_m\to 1$.
Since the function $u$ is smooth, $\mathbb{X}_{+}(u_m)$ must tend to zero
in the $L^{2}$-topology as $m\to\infty$. Hence $a_{m+1}\to 0$
as $k\to\infty$ which in turns implies that $u_{N}=0$ for any $N\geq N_0$ 
since $\left(\prod_{j=0}^{\infty}d_{m_0+2j}\right)$ is a finite constant.

Thus $u$ has finite degree as desired also for $\dim M=3$.
\end{proof}

\section{Twisted CKTs and ray transforms} \label{sec_twisted_ckts_raytransform}

In Section \ref{sec_finite_degree} we proved the finite degree result, Theorem \ref{thm:finitedegree_higgs}. In this section we give the easy argument that improves this result in cases where there are no nontrivial twisted conformal Killing tensors.

Recall from the introduction that the absence of nontrivial twisted CKTs means that any $u \in \Omega_m$, $m \geq 1$, satisfying $\mathbb{X}_+ u = 0$ (with $u|_{\partial(SM)} = 0$ in the boundary case) must be identically zero.

\begin{Theorem} \label{thm_finitedegree_notwistedckts}
Let $(M,g)$ be a negatively curved compact manifold with or without boundary. Assume 
that the boundary is strictly convex if $\pl M\not=\emptyset$. Let $(\mc{E},\nabla^{\mc{E}})$ be a Hermitian bundle with Hermitian connection and let $\Phi$ be a skew-Hermitian Higgs field. Suppose that $f \in C^{\infty}(SM;\mc{E})$ has degree $m \geq 0$, and that $u \in C^{\infty}(SM;\mc{E})$ (with $u|_{\partial(SM)} = 0$ in the boundary case) solves the equation 
$$
(\mathbb{X}+\Phi)u = -f \text{ in } SM.
$$
If $u$ has finite degree, and if there are no nontrivial twisted CKTs, then $u$ has degree $\max\{m-1,0\}$. Furthermore, if $m=0$, then one has $u=0$ in the boundary case and $u \in \mathrm{Ker}(\mathbb{X}_+|_{\Omega_0})$ in the closed case.
\end{Theorem}
\begin{proof}
Let first $m \geq 1$, and let $l$ be the largest integer for which $u_l$ is nonzero. The claim is that $l \leq m-1$, so we argue by contradiction and assume that $l \geq m$. Looking at the degree $l+1$ Fourier coefficients in the identity $(\mathbb{X}+\Phi)u = -f$, we obtain that 
$$
\mathbb{X}_+ u_l = 0.
$$
Since there are no nontrivial twisted CKTs (note that $u_l|_{\partial(SM)} = 0$ in the boundary case since $u|_{\partial(SM)} = 0$), we have $u_l=0$. This contradicts the fact that $u_l$ was the largest nonzero Fourier coefficient.

In the case $m=0$ the above argument shows that $u = u_0$, and the equation becomes 
$$
(\mathbb{X}+\Phi)u_0 = -f_0.
$$
Taking degree $1$ Fourier coefficients gives $\mathbb{X}u_0 = 0$. In the boundary case 
we have $u_0|_{\partial(SM)} = 0$ and by Proposition \ref{Rpm0} this implies that $u_0=0$ if the curvature is negative 
and $\pl M$ is strictly convex.
In the closed case the equation $\mathbb{X}u_0 = 0$ means that $u_0 \in \mathrm{Ker}(\mathbb{X}_+|_{\Omega_0})$.
\end{proof}

The injectivity result in the boundary case, Theorem \ref{mainthm_boundary1}, will require the absence of twisted conformal Killing tensors vanishing on the boundary. In other words we would like to prove:

\begin{Theorem}
Let $(M,g)$ be a Riemannian manifold and $(\mc{E},\nabla^{\mc{E}})$ a Hermitian bundle with connection. Let $\Gamma$ be a hypersurface.  Assume there is
$u\in\Omega_m$ with $\mathbb{X}_{+}u=0$ and $u|_{\pi^{-1}\Gamma}=0$. Then $u=0$.
\label{theorem:ACKT}
\end{Theorem}
\begin{proof}
By a connectedness argument, the proof reduces to a local statement and thus it suffices to consider the case of a trivial bundle $SM \times \C^n$ 
with a connection $\nabla^{\mc{E}}=d+A$ for some connection $1$-form $A$ (with values in skew-Hermitian matrices). 
The operator $\mathbb{X}_+$ can then be written as $\mathbb{X}_+=X_++A_+$ where $X_+$ is the usual conformal Killing 
operator in the trivial bundle $SM \times \C^n$ acting diagonally, and $A_+$ is an endomorphism acting on $SM \times \C^n$ 
(an operator of order $0$). 
For  $A=0$ this theorem was proved in \cite{DS} and we shall use their approach for Step (2) below.
The proof splits in two:
\begin{enumerate}
\item First show that a solution to $\mathbb{X}_{+}u=0$ is determined by the $N$-jet of $u$ at a point, for a suitable $N$.
\item Show that if $u|_{\pi^{-1}\Gamma}=0$, then $u$ vanishes to infinite order at any point in $\pi^{-1}\Gamma$.
\end{enumerate}

These two steps correspond to Theorem 1.1 and 1.3 in \cite{DS} respectively. Both items will follow from results in the literature as we now explain. If we think of $u$ as a trace free symmetric $m$-tensor then $(X_{+}+A_+)u=0$
is equivalent to $\mc{P}(Du+\mc{S}(A\otimes u))=0$ where $D=\mc{S}\nabla$ is the usual conformal Killing operator 
(see Section \ref{Identifi}), $\mc{P}$ the projection on trace-free symmetric tensors and $\mc{S}$ denotes symmetrization as in Section \ref{Identifi}. 
Since $\mc{P}(D+\mc{S}(A\otimes \cdot))$ and $\mc{P}D$ have the same principal symbol, Theorem 3.6 in 
\cite{C} implies directly that any solution $u$ to $\mc{P}(Du+\mc{S}(A\otimes u))=0$ 
is determined by the $N$-jet of $u$ in one point, for some suitable $N$.

We are left with showing (2) and for this we can employ exactly the same proof as in \cite[Lemma 4.1]{DS}
which is the main lemma showing item (2) for the case $A=0$.
To this end, we note that once
equations (4.1) and (4.2) in \cite[Lemma 4.1]{DS} are established, the rest the proof runs undisturbed based on these two equations. The proof is by induction and equation (4.1) in \cite[Lemma 4.1]{DS} is the induction assumption which just claims that derivatives up to order $k$ vanish.

But we claim that $\mc{P}(Du+\mc{S}(A\otimes u))=0$ leads exactly to the {\it same} equation (4.2)  in \cite[Lemma 4.1]{DS} even when $A$ is not zero. To see this observe that  $\mc{P}(Du+\mc{S}(A\otimes u))=0$ is equivalent to
\[Du+\mc{S}(A\otimes u)=\mc{S}(g\otimes v) \]
where $v\in \Theta_{m-1}$.  In coordinates and using the notation from \cite{DS} the term $\mc{S}(A\otimes u)$ is given by
\[\frac{1}{m+1}(A_{i_{1}}u_{i_{2}\dots i_{m}}+A_{i_{2}}u_{i_{1}i_{3}\dots i_{m}}+\dots+A_{i_{m+1}}u_{i_{1}\dots i_{m}}).   \]
The coordinates $(x^{1},\dots,x^{n-1},y)$ are chosen so that $y=0$ defines $\Gamma$ and 
they are normal geodesic coordinates (i.e. $g_{in}=\delta_{in})$. 
But, once we apply the operator $\frac{\partial^{k}}{\partial y^{k}}|_{y=0}$ to this expression it vanishes so the equation that we obtain is exactly the same as equation (4.2) in \cite[Lemma 4.1]{DS} and we are done.
\end{proof}

\section{Regularity for solutions of the transport equation} \label{sec_regularity}

\vspace{10pt}

\noindent \ref{sec_regularity}.1. {\bf Geometric setup and geodesic flow.} \label{Geosetup}
We consider a smooth compact Riemannian manifold $(M,g)$ with strictly convex boundary $\pl M$ 
and we assume that the sectional curvatures of $g$ are negative. We let $X$ be the geodesic vector field of $g$ on $SM$. For convenience of notations and technical purpose, we will extend the vector field $X$ to a larger manifold with boundary in a way that it has complete flow and for that purpose we follow very closely the method explained in \cite[Sec. 2.1]{G14b}. 
We can extend $M$ to a smooth compact manifold $\hat{M}$ with boundary by adding a very small collar to $M$ and extend $g$ so that $\hat{M}\setminus M$ has a foliation by strictly convex hypersurfaces, the boundary $\pl \hat{M}$ is strictly convex, 
and $g$ has negative curvature. The geodesic vector field $\hat{X}$ for $g$ on $S\hat{M}$ coincides with that of 
$g$ when restricted on $SM$.  Each trajectory leaving $SM$ never comes back to $SM$ and hits $\pl S\hat{M}$ in finite time. We multiply $\hat{X}$ by a non-negative function $\rho_0 \in 
C^\infty(\hat{M})$ which is a function of the geodesic distance to $\pl M$ in $\hat{M}\setminus M$, 
vanishing only at $\pl\hat{M}$, at first order, and equal to $1$ in a neighborhood of $M$. If  
$\pi: S\hat{M}\to \hat{M}$ is the natural projection, the flow of $\pi^*(\rho_0) \hat{X}$ is complete on $S\hat{M}$, and the intersection of a flow line 
for $\pi^*(\rho_0) \hat{X}$ with $SM$ is exactly the flow line of $X$ in $SM$. By abuse of notation, we denote
the extension $\pi^*(\rho_0) \hat{X}$ of $X$ to $S\hat{M}$ by $X$, 
this allows us to consider $X$ as a vector field with a complete flow 
$\varphi_t: S\hat{M}\to S\hat{M}$. We also take an intermediate manifold $M_e\subset \hat{M}$ containing $M$ with the same properties as $M$, with $\rho_0=1$ on $M_e$. By our choice of $M_e$, the largest time that a flow trajectory spends in $SM_e\setminus SM$ is finite and denoted 
\begin{equation}\label{defL}
L :=\sup \{ t\geq 0;  \exists y\in SM_e\setminus SM, \forall s\in [0,t], \varphi_s(y)\in SM_e\setminus SM \} <\infty.
\end{equation}

We now describe properties of the geodesic flow in negative curvature; 
we refer to Section 2 of \cite{DG14} and to Sections 2.2, 2.3 in \cite{G14b} for more details. 
For each point $(x,v)\in SM$, we define the time of escape from $SM$ along
the forward ($+$) and backward ($-$) trajectories: \begin{equation}\label{escapetime} 
\begin{gathered}
\ell_+ (x,v)=\sup\, \{ t\geq 0; \varphi_{t}(x,v)\in SM\}\subset [0,+\infty],\\
\ell_- (x,v)=\inf\, \{ t\leq 0; \varphi_{t}(x,v)\in SM\}\subset [-\infty,0].
\end{gathered}
\end{equation}
Then we define the incoming ($-$) and outgoing ($+$) tails in $SM$ by 
\[
\Gamma_\mp =\{(x,v)\in SM;  \ell_\pm(x,v)=\pm \infty\}=\bigcap_{t\geq 0}\varphi_{\mp t}(SM)
\]
and the trapped set for the flow on $SM$ is the closed (flow-invariant) subset of $SM^\circ$ 
\begin{equation}\label{defofK} 
K:=\Gamma_+\cap \Gamma_-=\bigcap_{t\in \mathbb{R}}\varphi_t(SM).
\end{equation}
Here $\Gamma_-$ is the stable manifold of $K$ and $\Gamma_+$ is the unstable manifold of $K$ for the flow.
Since the curvature is negative, the set $K$ is a hyperbolic set in the sense of dynamical systems, i.e. it has a decomposition of the form 
\[
T_p(SM) = E_0(p) \oplus E_s(p) \oplus E_u(p), \quad \forall p\in K
\] 
which is continuous in $p$ and invariant by the flow, where $E_0=\mathbb{R} X$ and $E_s$ and  $E_u$ are stable and unstable bundles 
as in \eqref{stable/unstable}. The bundle $E_s$ extends continuously to a bundle called $E_-$ over $\Gamma_-$ 
and $E_u$ to a bundle called $E_+$ over $\Gamma_+$ (the fibers are simply the tangent spaces to each stable/unstable leaf); the differential of the forward flow is uniformly contracting on $E_-$ and uniformly expanding on $E_+$. As in 
\cite[Lemma 2.10]{DG14}, there are dual subbundles $E_\pm^*\subset T_{\Gamma_\pm}^*(SM)$ over $\Gamma_\pm$ satisfying 
\[E_+^*(E_+\oplus E_0)=0, \quad E_-^*(E_-\oplus E_0)=0. \]
Finally, by Proposition 2.4 in \cite{G14b}, if $g$ is negatively curved there exists $Q<0$ so that
\begin{equation}\label{Vt}
\begin{gathered}
V(t)=\mc{O}(e^{Q|t|}) \textrm{ where } \\
V(t):={\rm Vol}(\{ y\in SM;\, \varphi_s(y)\in SM \textrm{ for } |s|\in[0,|t|], st>0\}).
\end{gathered}\end{equation}
The volume is with respect to the Liouville measure $d\mu$. The boundary $\pl(SM)$ has a natural measure 
$d\mu_\nu$ which in local coordinates $(x,v)$ with $x\in \pl M$ and $v\in S^{d-1}$ 
is given by $d\mu_\nu=|\cjg v,\nu\cjd_g|d{\rm vol}_{\pl M}dv_{S^{n-1}}$, and we shall always use this measure when we integrate on $\pl(SM)$.
In particular we have that ${\rm Vol}_{SM}(\Gamma_+\cup \Gamma_-)=0$ and 
thus also ${\rm Vol}_{\pl(SM)}(\Gamma_\pm\cap \pl(SM))=0$ using that $X$ is transverse to $\pl M$ near $\Gamma_\pm$ (see \cite[Section 2.4]{G14b} for details).

\vspace{12pt}

\noindent \ref{sec_regularity}.2. {\bf The operator generating attenuated transport and its resolvent.}
Consider a Hermitian vector bundle $\mc{E}$ on $SM$ (with a Hermitian product $\cjg\,\cdot\,,\,\cdot\,\cjd_{\mc{E}}$), 
and let $\nabla$ be a Hermitian connection, i.e.\ 
\begin{equation} 
V\cjg w,w'\cjd_{\mc{E}}=\cjg \nabla_Vw,w'\cjd_\mc{E}+\cjg w,\nabla_Vw'\cjd_\mc{E}
\end{equation}
for any smooth sections $w,w'$ of $\mc{E}$ and any vector field $V$ on $SM$. Now, 
let us take $\Phi \in C^{\infty}(SM; {\rm End}_{\rm sk}(\mc{E}))$ a skew-Hermitian potential and we extend $\mc{E}$, 
$\nabla$ and $\Phi$  to 
$S\hat{M}$ in a smooth fashion. Note that later $\mc{E},\nabla,\Phi$ will be taken to be pull-back of bundles, connections and Higgs fields on the base manifold $M$, in order to use Pestov identities, but in this section this is not needed.
Let $\mathbb{X}=\nabla_X$ be the first order differential operator acting on sections of $\mc{E}$ over $S\hat{M}$
already introduced in (\ref{defbfX}). It satisfies:
for all $f\in C^\infty(S\hat{M};\mc{E})$, 
$\psi\in C^\infty(S\hat{M})$ and $f'\in C_c^\infty(SM^\circ;\mc{E})$,
\begin{equation}\label{bfXu} 
\mathbb{X} (\psi f)=(X\psi)f+\psi(\mathbb{X}f), \quad  \cjg \mathbb{X}f,f'\cjd_{L^2(S\hat{M};\mc{E})}=- \cjg f,\mathbb{X}f'\cjd_{L^2(S\hat{M};\mc{E})}.
\end{equation} 
where the $L^2$ space is defined with respect to the Liouville measure on $SM$. Let us define
\[ P:=-\mathbb{X}-\Phi \]
acting on smooth sections of $\mc{E}$, which is formally skew-adjoint when restricted to the space 
$C_c^\infty(SM^\circ, \mc{E})$. 
Its propagator $U(t):=e^{-tP}$ is the operator which solves the equation $\pl_tU(t)f=-PU(t)f$ for all $f\in C_c^\infty(S\hat{M}^\circ, \mc{E})$ with $U(0)={\rm Id}$. Here $U(t)$ is well defined as the solution of a non-characteristic first order ODE. If $\mc{E}\simeq \mathbb{C}$ is trivial with the trivial connection and $\Phi=0$, then $U(t)f=f\circ \varphi_t$.
Note that the first property of \eqref{bfXu} on $S\hat{M}$ implies 
\begin{equation}
  \label{support}
U(t)(\psi f)=(\psi\circ\varphi_t)U(t)f,\quad
\forall \psi\in C_c^\infty(S\hat{M}),\ \forall f\in C_c^\infty(S\hat{M};\mc{E}).
\end{equation}
In particular, for $f\in C_c^\infty(SM^\circ_e;\mc{E})$, $U(t)f$ has support intersecting $SM_e\setminus SM$ if and only 
if $\varphi_{-t}({\rm supp}(f))\cap (SM_e\setminus SM)\not=\emptyset$.
For all $f\in L^2(SM)$ such that ${\rm supp}(U(t)f)\subset SM$, one has (using density of $C_c^\infty(SM^\circ)$ in $L^2(SM)$)
\begin{equation}\label{unitary} 
||U(t)f||_{L^2}=||f||_{L^2},
\end{equation}
this follows directly from the fact that $\mathbb{X}$ is formally skew-adjoint in $SM$ and 
that $\Phi$ is skew-Hermitian over $SM$. 
Define 
\[\mc{T}_{\pm}(t):=\{ y\in SM;\, \varphi_{\pm s}(y)\in SM \textrm{ for } s\in[0,t]\}.\]
Then by definition of the constant $L$ in \eqref{defL} and by \eqref{unitary}, we can write for $f\in L^2(SM_e;\mc{E})$ and $t>L$
\begin{equation}\label{estimU(t)u} 
||U(t)f||_{L^2(SM_e)}=||U(L)U(t-L)(f.1_{\mc{T}_+(t-L)})||_{L^2(SM_e)}\leq C_L||f||_{L^2(SM_e)}.
\end{equation}
with $C_L=||U(L)||_{L^2(SM_e)\to L^2(SM_e)}$. 
We obtain
\begin{Lemma}
For ${\rm Re}(\la) > 0$, the resolvents $R_\pm (\la):=(P\pm \la)^{-1}$ are bounded 
as maps on $L^2(SM_e;\mathcal E)$ 
and given in terms of the propagator by the formula
\begin{equation}
 \label{formulares}
R_\pm(\la)f=\int_0^{\pm \infty} e^{\mp \la t}U(t)f \,dt.
\end{equation}
They satisfy $(P\pm \la)R_\pm(\la)f=f$ in the distribution sense in $SM_e$, and  if $f\in C^0(SM;\mc{E})$, then  
$R_\pm(\la)f$ is continuous near $\pl_\pm(SM)$ and  
$(R_\pm(\la)f)|_{\pl_\pm(SM)}=0$. 
\end{Lemma}
\begin{proof}
Let ${\rm Re}(\la)>0$. Then by Cauchy-Schwartz and \eqref{estimU(t)u}
\[ \int_0^{\mp \infty} e^{\pm {\rm Re}(\la) t}||U(t)f||_{L^2(SM_e)}dt\leq C_{\la,L} ||f||_{L^2(SM_e)}\]
for some $C_{\la,L}>0$ depending on ${\rm Re}(\la)$ and $L$, thus $R_\pm(\la)$ is bounded on $L^2$.  The other properties are straightforward: the continuity of 
$R_\pm(\la)f$ near $\pl_\pm (SM)$ is just ODE regularity and $(R_\pm(\la)f)|_{\pl_\pm(SM)}=0$
follows from $\ell_\pm|_{\pl_\pm(SM)}=0$. For more details, see 
\cite[Lemma 4.1]{G14b} where it is done for $\mc{E}=\mathbb{C}$ and $P=-X$.
\end{proof}

We want to define a right inverse for $P$ and thus we let $\la\to 0$ to define $R_\pm(0)$. The problem is that 
this is not bounded on $L^2(SM_e;\mc{E})$, but arguing like in Propositions 4.2-4.4 of \cite{G14b}, we can prove, 
using the properties on the trapped set, that these operators make sense when acting on $L^p$ spaces, and Sobolev spaces of positive order. We refer to \cite{Ho} for definitions and properties of wavefront set of distributions (which is denoted by ${\rm WF}$ below).
\begin{Proposition}\label{Rpm0}
The resolvent $R_\pm(\la)$ extends continuously to ${\rm Re}(\la)\geq 0$ as a family of bounded operators
for $s\in (0,1/2)$ and any $p<\infty$
\[R_\pm(\la): H_0^s(SM_e;\mc{E})\to H^{-s}(SM_e;\mc{E}) , \quad 
R_\pm(\la): L^\infty(SM_e;\mc{E})\to L^p(SM_e;\mc{E}) \]
that satisfies  $(P\pm \la)R_\pm(\la)f=f$ in the distribution sense in $SM_e$, and for $f\in C^0(SM_e;\mc{E})$ the expression 
\eqref{formulares} holds true also in ${\rm Re}(\la)\geq 0$ as an element in $L^p(SM_e;\mc{E})$. 
If $f\in C^\infty(SM; \mc{E})$ is extended by $0$ outside $SM$, 
the section $u_\pm:=R_\pm(0)f$ is smooth in $SM\setminus \Gamma_\mp$ and its 
wavefront set over $SM^\circ$ is
\begin{equation}\label{WFsetupm}
{\rm WF}(u_\pm)\cap T^*SM^\circ\subset E_\mp^*,
\end{equation}
the restriction $u_\pm|_{\pl(SM)}$ makes sense as a distribution satisfying 
\begin{equation}\label{WFsetupmpl}
\forall p<\infty,\,\,  u_\pm|_{\pl(SM)}\in L^p(\pl(SM)),\quad 
u_\pm |_{\pl_\pm(SM)}=0. 
\end{equation}
Finally, the restriction of $u_\pm$ to $SM$ is the only  $L^1(SM;\mc{E})\cap C^\infty(SM\setminus \Gamma_\mp;\mc{E})$  section
which satisfies $Pu_\pm=f$ in $SM$ in the distribution sense and $u_\pm|_{\pl_\pm(SM)}=0$.
\end{Proposition}
\begin{proof}
For any $\delta>0$, the resolvents $R_\pm(\la)$ admits a meromorphic extension in ${\rm Re}(\la)>-\delta$ as 
bounded operators $R_\pm(\la): H_0^s(SM_e; \mc{E})\to H^{-s}(SM_e;\mc{E})$ for $0<s<C\delta$ for some $C>0$ depending  on the Lyapunov exponents. The meromorphic extension and boundedness is proved in Lemmas 4.2-4.4 of \cite{DG14} (see also the remark after Lemma 4.2 in \cite{DG14} for the sharp Sobolev exponent in the case where $P$ is formally skew-adjoint near the trapped set $K$). The fact that $R_\pm (\la)$ is continuous in ${\rm Re}(\la)\geq 0$ follows essentially from the proof of
Proposition 4.2 in \cite{G14b}: using \eqref{support} we have pointwise estimates for $u_+(\la;y):=(R_+(\la)f)(y)$
\[ ||u_+(\la;y)||_{\mc{E}} \leq \int_{0}^\infty \chi(\varphi_t(y))||U(t)f(y)||_{\mc{E}}dt\leq 
C ||f||_{L^\infty} \int_{0}^\infty \chi(\varphi_t(y)) dt\]
if $\chi\in C_c^\infty(S\hat{M}^\circ)$ is non-negative, equal to $1$ on $SM_e$ and supported in a very small neighborhood of $SM_e$ (here $C>0$ depends only on $L$). Then the proof of \cite[Prop. 4.2]{G14b} can be applied verbatim and using \eqref{Vt}, we obtain that for each $p<\infty$ there is $C>0$ such that for all ${\rm Re}(\la)\geq 0$
\[ ||u_+(\la)||_{L^p(SM_e;\mc{E})}\leq C||f||_{L^\infty(SM_e;\mc{E})}.\]
The wave-front set properties \eqref{WFsetupm} and \eqref{WFsetupmpl} are obtained exactly as in (the proof of)\cite[Prop 5.5]{G14b}: they are direct consequences of the analysis in \cite[Lemma 4.5]{DG14} of the wavefront set of the Schwartz kernel of the resolvent $R_\pm(0)$ and the composition of wavefront sets given by \cite[Th. 8.2.13]{Ho}. The vanishing of $u_\pm$ on $\pl_\pm(SM)$
is easy and holds the same way as for ${\rm Re}(\la)>0$. The fact that $u_\pm|_{SM}$ is the only $L^1$ solution of $Pu_\pm=f$ in $SM$ vanishing at $\pl_\mp SM$ and smooth outside $\Gamma_\mp$ is also clear: the difference of two such solutions would be an $L^1$ section in $\ker P$ that is smooth in $SM\setminus \Gamma_\mp$ and such sections are uniquely determined in $SM\setminus \Gamma_\mp$ from parallel transports of elements of $\mc{E}|_{\pl_\mp SM\setminus \Gamma_\mp}$ along flow trajectories of $X$, and are thus determined in a set of full Liouville measure 
by their value at $\pl_\mp SM\setminus \Gamma_\mp$, using that ${\rm Vol}(\Gamma_+\cup \Gamma_-)=0$.
\end{proof}
In fact, the exponential decay \eqref{Vt} implies that 
\[\exists \alpha>0,\,\, f\in C^\alpha(SM_e;\mc{E}) \Longrightarrow \exists s>0,\,\, R_\pm (0)f\in H^s(SM_e;\mc{E}), \]
see \cite[Prop. 4.2]{G14b} for the argument.

Before stating the next corollary we recall that $I_{\nabla,\Phi}$ denotes the attenuated ray transform defined in (\ref{eq:defat}).

\begin{Corollary} \label{smoothness}
Assume that $I_{\nabla,\Phi}f=0$ with $f\in C^\infty(SM;\mc{E})$, then there exists a unique $u\in C^\infty(SM;\mc{E})$ such that
$(\mathbb{X}+\Phi)u=-f$ and  $u|_{\pl(SM)}=0$.
\end{Corollary}
\begin{proof}
Let $u_\pm:=R_\pm(0)f$ so that $Pu_\pm=f$ in $SM$ vanishing on $\pl_{\pm }(SM)$. Since $I_{\nabla,\Phi}f=0$ and $\Gamma_\pm\cap \pl_\pm(SM)$ has measure $0$, we have $u_+|_{\pl_-(SM)}=I_{\nabla,\Phi}f=0$ as an 
$L^1(\pl_-(SM))$ function, thus by the last statement of Proposition \ref{Rpm0} we obtain $u_-=u_+$. By Proposition \ref{Rpm0}, this implies that $u_+=u_-\in C^\infty(SM\setminus K)$ and using \eqref{WFsetupm} together with the fact that $E_-^*\cap E_+^*$ is the zero section over the trapped set $K$  we conclude that 
\[ {\rm WF}(u_+)\subset {\rm WF}(u_+)\cap {\rm WF}(u_-)=\emptyset.\]
This shows that $u:=u_+$ is smooth in $SM$ and  $(\mathbb{X}+\Phi)u=-f$ with $u|_{\pl(SM)}=0$.
\end{proof}

We can now easily prove the main injectivity result for the attenuated ray transform in the boundary case:

\begin{proof}[Proof of Theorem \ref{mainthm_boundary1}]
Suppose $f \in C^{\infty}(SM ; \mc{E})$ has degree $m$ and $I_{\nabla, \Phi} f = 0$. By Corollary \ref{smoothness} there is a unique $u \in C^{\infty}(SM ; \mc{E})$ with $(\mathbb{X} + \Phi) u = -f$ and $u|_{\partial(SM)} = 0$. Theorem \ref{thm:finitedegree_higgs} implies that $u$ has finite degree. Since there are no nontrivial twisted CKTs in the boundary case (Theorem \ref{theorem:ACKT}), Theorem \ref{thm_finitedegree_notwistedckts} implies that $u$ has degree $m-1$.
\end{proof}

\vspace{12pt}

\noindent \ref{sec_regularity}.3. {\bf Scattering operator for $\mathbb{X}+\Phi$.}
In this section, we shall describe the regularity of the solutions $u\in L^1(SM; \mc{E})$ of the transport equation 
\begin{equation}\label{transportP}
(\mathbb{X}+\Phi)u=0  \textrm{ in distribution sense in }SM,\quad u|_{\pl_-(SM)}=\omega
\end{equation}
where $\omega\in C^\infty(\pl_-(SM);\mc{E})$. Clearly, if $\omega$ is supported in $\pl_-(SM)\setminus \Gamma_-$, the solution $u$ is unique and smooth, just as in the non-trapping case, and its support is disjoint from $\Gamma_+\cup \Gamma_-$. This allows to define the \emph{scattering operator} for $P$ 
\begin{equation}\label{CAPhi}
\mc{S}^{\nabla,\Phi}: C_c^\infty(\pl_-(SM)\setminus \Gamma_-; \mc{E})\to C_c^\infty(\pl_+(SM)\setminus \Gamma_+;\mc{E}), \quad 
\mc{S}^{\nabla,\Phi}\omega:= u|_{\pl_+(SM)}.
\end{equation}
 We follow closely the results of section 4.3 in \cite{G14b}, 
in particular Proposition 4.6 of this article.
\begin{Proposition}\label{regPu=0}
Let $(M,g)$ be a negatively curved  manifold with strictly convex boundary, let $\mc{E}$ be a Hermitian bundle with Hermitian connection on $SM$ and let $\Phi$ be a smooth skew-Hermitian potential. 
Then for each $\omega\in C^\infty(\pl_-(SM);\mc{E})$, there is a unique solution $u$ of 
\eqref{transportP} in $L^1(SM;\mc{E})$, which in addition is in $L^p(SM;\mc{E})$ 
for all $p<\infty$ and in $C^\infty(SM\setminus \Gamma_+; \mc{E})$. The map $\omega\mapsto u$ is 
continuous as operator $L^\infty(\pl_-SM;\mc{E})\to L^p(SM;\mc{E})$ for all $p<\infty$. 
Moreover, the operator $\mc{S}^{\nabla,\Phi}$ of \eqref{CAPhi} extends as a unitary operator 
\[\mc{S}^{\nabla,\Phi}: L^2(\pl_-(SM);\mc{E})\to L^2(\pl_+(SM);\mc{E})\] 
and if $\mc{S}^{\nabla,\Phi}\omega\in C^\infty(\pl_+(SM);\mc{E})$ with $\omega\in C^\infty
(\pl_-(SM);\mc{E})$, 
then $u\in C^\infty(SM;\mc{E})$.
\end{Proposition}
\begin{proof} Let us first show that $\mc{S}^{\nabla,\Phi}$ extends as a unitary map, we follow the proof of Lemma 3.4 in 
\cite{G14b}. If $\omega_1,\omega_2$ are in $C_c^\infty(\pl_-(SM);\mc{E})$ and $u_1$ and $u_2$ are the 
$C_c^\infty(SM\setminus (\Gamma_-\cup\Gamma_+); \mc{E})$ solutions of $Pu_i=0$ with $u_i|_{\pl_-(SM)}=\omega_i$, then  
\[ \begin{split}
0=&\int_{SM} \cjg Pu_1,u_2\cjd_\mc{E}+\cjg u_1,Pu_2\cjd_\mc{E} d\mu=
-\int_{SM} \cjg \mathbb{X}u_1,u_2\cjd_\mc{E}+\cjg u_1,\mathbb{X}u_2\cjd_\mc{E} d\mu\\
=&- \int_{SM} X(\cjg u_1,u_2\cjd_\mc{E})d\mu=\int_{\pl_-(SM)}\cjg \omega_1,\omega_2\cjd_{\mc{E}}d\mu_{\nu}
-\int_{\pl_+(SM)}\cjg \mc{S}^{\nabla,\Phi}\omega_1,\mc{S}^{\nabla,\Phi}\omega_2\cjd_{\mc{E}}d\mu_{\nu}.
\end{split}\]
This implies that $\mc{S}^{\nabla,\Phi}$ is a unitary operator for the $L^2$ product. The proof of the existence of an 
$u\in L^1(SM;\mc{E})$ solving \eqref{transportP} is very similar to the proof of Proposition 4.6 of \cite{G14b}, thus we just sketch the argument. It suffices to assume that $\omega$ is supported near $\Gamma_-$ 
as the case where $\omega$ has support not intersecting $\Gamma_-$ is standard. Since the trapped set is at positive distance from $\pl(SM)$, we can construct $\tilde{u}_-\in C^\infty(SM_e;\mc{E})$ so that $\tilde{u}_-|_{\pl_-(SM)}=\omega$ and 
${\rm supp}(P\tilde{u}_-)\cap SM_e^\circ \subset SM^\circ \setminus \Gamma_+$. Then we set 
$u=\tilde{u}_--R_-(0)(P\tilde{u}_-)$ which, by Proposition \ref{Rpm0}, is an $L^p(SM_e;\mc{E})$ 
section for all $p<\infty$, smooth outside $\Gamma_+$ and solves \eqref{transportP} in $SM$. The solution is clearly unique since it is determined uniquely by $\omega$ at each point $(x,v)\in SM$ so that $|\ell_-(x,v)|<\infty$ (that is in $SM\setminus \Gamma_+$) which is a set of full measure in $SM$.
By construction and continuity of $R_-(0)$ in Proposition \ref{Rpm0}, the map $\omega\mapsto u$ is bounded as map 
$L^\infty\to L^p$.
 It remains to show that $u$ is in fact smooth if $\mc{S}^{\nabla,\Phi}\omega\in C^\infty(\pl_+(SM);\mc{E})$. The proof follows basically the proof of 2) in \cite[Proposition 4.6]{G14b}. First, it is clear that $u\in C^\infty(SM\setminus K;\mc{E})$ since the solution there is obtained from composition of $\omega$ in forward and backward time by parallel transport; this can be viewed also in terms of propagation of singularities for principal type operators \cite[Prop. 2.5]{DZ13}. Then, to show that $u$ is in any positive Sobolev space $H^s$ near $K$, we argue exactly as in the end of the proof of \cite[Proposition 4.6]{G14b}: 
using that $P\tilde{u}_-$ is smooth and compactly supported in $SM_e^\circ$ and \cite[Prop. 6.1]{DG14}\footnote{In the statement of Proposition 6.1 of \cite{DG14}, the operator is chosen (for notational convenience) to be the flow vector field $X$ and the bundle $\mc{E}$ is the trivial bundle $\mathbb{C}$, but the analysis of the resolvent $R_\pm(\la)=(P-\la)^{-1}$ is done for general bundles and contains the case of operators 
$P=-\mathbb{X}+\Phi$ with $\mathbb{X}$ as in our paper and $\Phi$ any skew-Hermitian potential, therefore all the statements of that Proposition apply to our case as well.},  
there is a pseudo differential operator $A_-$ of order zero which is microsupported in a conic neighborhood of $E_-^*\subset T^*SM$ and elliptic in a neighborhood of $E_-^*$ so that $A_-R_-(0)(P\tilde{u}_-)\in H^s(SM;\mc{E})$ for all $s>0$. As a consequence the wavefront set
of $u$ is disjoint from a conic neighborhood of $E_-^*$. By elliptic regularity, we also deduce that 
${\rm WF}(u)\subset E_+^*\oplus E_-^*$. Next, we use propagation of singularities \cite[Prop. 2.5]{DZ13}: by \cite[Lemma 2.10]{DG14}, the trajectories of the Hamiltonian vector field of the principal symbol $p(\xi)=\xi(X)$ of $X$ contained in the energy level $p^{-1}(0)=E_+^*\oplus E_-^*$ are either contained in $E_+^*$, or converge to $E_-^*$ or reach $T^*_{\pl_-(SM)}SM$ in backward time, we conclude that ${\rm WF}(u)\subset E_+^*$ since we know that $u$ is regular in a neighborhood of 
$E_-^*\cup T^*_{\pl_-(SM)}SM$. Finally, if $\mc{S}^{\nabla,\Phi}(\omega)$ is smooth, we can see as above that $u=\tilde{u}_+-R_+(0)P\tilde{u}_+$ for some $\tilde{u}_+$ that is smooth and supported near $\pl_+(SM)$, $\tilde{u}_+|_{\pl_+(SM)}=\mc{S}^{\nabla,\Phi}\omega$ and $P\tilde{u}_+\in C_c^\infty(SM_e^\circ;\mc{E})$ has support not intersecting $\Gamma_-$. Applying the same reasoning as before with the backward flow, we deduce that there is a pseudo differential operator $A_+$ of order zero which is microsupported in a conic neighborhood of $E_+^*\subset T^*SM$, elliptic near $E_+^*$ and so that $A_+R_+(0)(P\tilde{u}_+)\in H^s(SM;\mc{E})$ for all $s>0$, which shows that $u$ is smooth.
\end{proof}

\section{Parallel transport and gauge equivalent connections}\label{scatteringequivalent}

We will now prove Theorem \ref{mainthm_boundary2}. The proof is  similar to \cite[Theorem 8.1]{PSU2} but we need to be careful about regularity issues when there are trapped geodesics.\\

\emph{Proof of Theorem \ref{mainthm_boundary2}}. 
Let $(\mc{E},\nabla^{\mc{E}})$ be a Hermitian bundle with Hermitian connection.
Any other Hermitian connection on $\mc{E}$ over $M$ can be written as $\nabla^{\mc{E}}+A$ for some skew-Hermitian 
connection $1$-form $A\in C^\infty(M;T^*M\otimes {\rm End}_{\rm sk}(\mc{E}))$. Notice that $A$ can also be viewed as an element in $C^\infty(SM;{\rm End}_{\rm sk}(\mc{E}))$ which has degree $1$ in $v$, by considering $(x,v)\mapsto A(x)(v)$ (or equivalently by  contracting $\pi^*A$ with $X$). 
We consider the bundle $\mc{F}:={\rm End}(\mc{E})$ over $M$, which is Hermitian with Hermitian product $\cjg U,W\cjd:={\rm Tr}(UW^*)$ where the adjoint is taken using the Hermitian product on $\mc{E}$. The bundle $\mc{F}$ has a natural Hermitian connection $\nabla^{\mc{F}}$ given by (recall that $\nabla^{\mc{E}}$ is the connection on $\mc{E}$)
\[ (\nabla_V^{\mc{F}} U)f:=\nabla^{\mc{E}}_V(Uf)-U(\nabla^{\mc{E}}_Vf) \]
where $f\in C^\infty(M;\mc{E})$ and $V$ any vector field on $M$. 
This bundle and connection pull-back to $SM$ via $\pi:SM\to M$, and we keep the same notations for the pull-back.
For a section $U\in C^\infty(M;\mc{F})$, we have $(\nabla^{\mc{F}}U)(x)(v)=(\nabla^{\mc{F}}_{X(x,v)}\pi^*U)$.
For $A\in C^\infty(M;T^*M\otimes {\rm End}_{\rm sk}(\mc{E}))$ a connection $1$-form and $\Phi$ a skew-Hermitian Higgs field, multiplication on the left by $A$ and by $\Phi$ on sections of $\mc{F}$ are skew-Hermitian linear maps with respect to the Hermitian structure on $\mc{F}$.  The connection $U\mapsto \nabla^{\mc{F}}U+AU$ is a Hermitian connection on $\mc{F}$. 
Note that $\pi^*A(X)(x,v)=A(x,v)$ if we identify $A$ with an element of degree $1$ in the fibers. 
By Proposition \ref{regPu=0} applied to the bundle $\mc{F}$, there is $U$ in 
$L^p(SM; \mc{F})$ for all $p<\infty$ and smooth outside $\Gamma_+$, which solves 
$$
(\nabla^{\mc{F}}_X+A +\Phi) U = 0 \text{ in } SM, \quad U|_{\partial_{-}(SM)}=\id.
$$
in the distribution sense in $SM^\circ$. We also notice that $U(x,v)$ is a unitary transformation of $\mc{E}_{(x,v)}$ for all $(x,v)\in SM\setminus \Gamma_+$, since $A+\Phi$ is skew-Hermitian on $\mc{E}$. 
The \emph{scattering data} corresponding to the pair $(A,\Phi)$ in $(M,g)$ is the map
$$
C^{A,\Phi}\in L^2(\pl_+SM; U(\mc{E})), \ \  C^{A,\Phi}:= U|_{\partial_{+}(SM)}
$$
where $U(\mc{E})$ is the unitary group of $\mc{E}$.
Knowing $S_g$ and $C^{A,\Phi}$ implies that one knows how vectors in $\mc{E}$ are parallel transported along maximal geodesics from $\partial M$ to $\partial M$ in the presence $A$ and $\Phi$. Indeed, if 
$f\in C^\infty(SM;\mc{E})$ is equal to $e\in \mc{E}_{y_0}$ at $y_0:=(x_0,v_0)\in\partial_-(SM)$ and 
$\nabla^{\mc{E}}_X f=0$, then one has $Uf(y_0)=e$ and $(\nabla^{\mc{E}}_X+A+\Phi)(Uf)=0$, thus the parallel transport  of $e$ along the geodesic $\cup_{t=0}^{\ell_+(y_0)}\varphi_t(y_0)$  in presence of $A,\Phi$ is given by 
$(Uf)(S_g(y_0))=C^{A,\Phi}f(S_g(y_0))$: this is determined only 
by $C^{A,\Phi}$ as a function of $(A,\Phi)$ since 
$f$ is depending only on $\nabla^{\mc{E}}$ and $X$ but not on $(A,\Phi)$. 

The scattering data has the gauge invariance
$$C^{Q^{-1}(\nabla^{\mc{F}}+ A )Q,Q^{-1} \Phi Q} = C^{A,\Phi} \quad \text{if $Q \in C^{\infty}(M,GL(\mc{E}))$ satisfies $Q|_{\partial M} = \id$}.
$$
It follows that from the knowledge of $C^{A,\Phi}$ one can only expect to recover $\nabla$ and $\Phi$ up to a gauge transformation via $Q$ which satisfies $Q|_{\partial M} = \id$. If $\nabla$ is a Hermitian connection and $\Phi$ is skew-Hermitian, the map $U$ and the scattering relation
$C^{A,\Phi}$ take values in $U(\mc{E})$ and the scattering relation remains unchanged under unitary gauge transformations which are the identity on the boundary.

We want to compare two connections $\nabla$ and $\tilde{\nabla}$. 
Take $A$ to be the skew-Hermitian connection $1$-form so that $\tilde{\nabla}=\nabla+\tilde{A}$. Let $\Phi$ and $\tilde{\Phi}$ be two skew-Hermitian Higgs fields.
We write $\mathbb{X}:=\nabla^{\mc{F}}_X$ as we did on $\mc{E}$.
As above, by Proposition \ref{regPu=0} applied to the bundle $\mc{F}$, there are $U$ and $\tilde{U}$ which are 
in $L^p(SM; \mc{F})$ for all $p<\infty$ and smooth outside $\Gamma_+$, which solve 
\[\begin{gathered} 
\mathbb{X}U+\Phi U=0 ,\quad U|_{\pl_-(SM)}={\rm Id},\\
\mathbb{X}\tilde{U}+\tilde{A}\tilde{U}+\tilde{\Phi} \tilde{U}=0 ,\quad \tilde{U}|_{\pl_-(SM)}={\rm Id},
\end{gathered}\]
and $U(x,v), \tilde{U}(x,v)$ are unitary transformations of 
$\mc{E}_{(x,v)}$ for all $(x,v)\in SM\setminus \Gamma_+$. Thus they are invertible on $SM\setminus \Gamma_+$
and the inverse $\tilde{U}^{-1}\in L^p(SM;\mc{F})$ for any $p$ 
since the matrix components of the inverse in a given local orthonormal basis of $\mc{E}$ 
are sums of products of matrix components of $\tilde{U}$ in the basis.
Now if $\chi_\eps\in C_c^\infty(\pl_-(SM)\setminus \Gamma_-)$ 
tends to $1$ pointwise and is uniformly bounded by $1$, then there is a flow invariant 
smooth function $\hat{\chi}_\eps$ (i.e. $X\hat{\chi}_\eps=0$) 
satisfying $\hat{\chi}_\eps|_{\pl_-(SM)}=\chi_\eps$, tending to $1$ pointwise in $SM$ and with 
$||\hat{\chi}_{\eps}||_{L^\infty}\leq 1$.
Let  $U^{\eps}:=\hat{\chi}_\eps U$ and $\tilde{U}_\eps^{-1}:=\hat{\chi}_\eps \tilde{U}^{-1}$, 
these are in $C_c^\infty(SM\setminus (\Gamma_-\cup \Gamma_+); \mc{F})$ and satisfy
\begin{equation}\label{approximation}
\begin{gathered} 
(\mathbb{X}+\Phi)U_\eps=0 ,\quad U_\eps|_{\pl_-(SM)}=\chi_\eps{\rm Id},\\
\mathbb{X}\tilde{U}_{\eps}^{-1} -\tilde{U}_{\eps}^{-1}\tilde{A}-\tilde{U}_{\eps}^{-1}\tilde{\Phi} =0 ,\quad \tilde{U}_{\eps}^{-1} |_{\pl_-(SM)}=\chi_\eps {\rm Id}.
\end{gathered}
\end{equation}
We have that $U_\eps \to U$ in $L^p$ for all $p<\infty$ and $\tilde{U}_{\eps}^{-1} 
\to \tilde{U}^{-1}$ in $L^p$ for all $p<\infty$ as $\eps\to 0$.
Consider the section $Q:=U\tilde{U}^{-1}$ and $Q_\eps:=U_\eps \tilde{U}_\eps^{-1}$. 
Then $Q_\eps$ is smooth and in $L^p(SM;\mc{E})$ for all $p<\infty$, and $Q_\eps\to Q$ in $L^p$ for all $p<\infty$ (by Lebesgue theorem). Now, using that the scattering data is the same for $(\nabla,\Phi)$ and $(\tilde{\nabla},\tilde{\Phi})$ we get $Q|_{\pl(SM)}={\rm Id}$ and we also have by \eqref{approximation}
\[ \mathbb{X}Q_\eps+\Phi Q_\eps-Q_\eps \tilde{A}-Q_\eps\tilde{\Phi}=0.\]
The equation holds as smooth functions and by pairing this equation with any $Y\in C_c^\infty(SM^\circ;\mc{F})$,
we can let $\eps\to 0$ to deduce that 
\[ \mathbb{X}Q+\Phi Q-Q\tilde{A}-Q\tilde{\Phi}=0, \quad U|_{\pl(SM)}={\rm Id}\]
in the distribution sense in $SM^\circ$. Introducing a new connection 
$\hat{A}R:=-R\tilde{A}$ on $\mc{F}$ and a Higgs field $\hat{\Phi}(R):=\Phi R-R\tilde{\Phi}$, we easily check (like in \cite[Theorem 8.1]{PSU2}) that they are a Hermitian connection and skew-Hermitian Higgs field on the bundle $\mc{F}$. By Proposition \ref{regPu=0} applied to this bundle and pair $(\hat{A},\hat{\Phi})$, 
we deduce that $Q$ is actually smooth in $SM$ since its boundary value at $\pl(SM)$ is smooth. Then the proof can be concluded just as in \cite[Theorem 8.1]{PSU2}: take $W:=Q-{\rm Id}$, then using $\nabla^{\mc{F}}{\rm Id}=0$, this solves 
\[ \hat{\nabla}_XW+\hat{\Phi}(W) = \tilde{A}+\tilde{\Phi}-\Phi, \quad W|_{\pl(SM)}=0 \]
and thus  $I_{\hat{\nabla},\hat{\Phi}}(\tilde{A}+\tilde{\Phi}-\Phi)=0$ in the bundle $\mc{F}$, which implies by Theorem \ref{mainthm_boundary1} that there exists $Z\in C^\infty(M;\mc{F})$ vanishing at $\pl M$ such that 
\[ \mathbb{X}Z-Z\tilde{A}+\Phi Z-Z\tilde{\Phi}=\tilde{A}+\tilde{\Phi}-\Phi\]
and $W=Z$ and $Q=Z+{\rm Id}$ gives the desired gauge equivalence.\qed

\section{Absence of twisted CKTs on closed surfaces} \label{sec_twistedckts_twodim}

In this section we prove Theorem \ref{main_thm_twistedckts_twodim} which gives a condition ensuring the absence of nontrivial twisted CKTs on closed Riemann surfaces. To explain this we recall some notation from \cite{PSU2, PSU1} and \cite[Appendix B]{PSU_hd} that is specific to two dimensions.

If $(M,g)$ is a closed oriented Riemannian surface, there is a global orthonormal frame $\{ X, X_{\perp}, V \}$ of $SM$ equipped with the Sasaki metric, where $X$ is the geodesic vector field, $V$ is the vertical vector field defined for $u \in C^{\infty}(SM)$ by 
$$
Vu(x,v) = \langle \vd u(x,v), iv \rangle
$$
where $iv$ is the rotation of $v$ by $90^{\circ}$ according to the orientation of the surface, and $X_{\perp} = [X, V]$. We define the Guillemin-Kazhdan operators \cite{GK}
$$
\eta_{\pm} = \frac{1}{2}(X \pm i X_{\perp}).
$$
If $x = (x_1, x_2)$ are oriented isothermal coordinates near some point of $M$, we obtain local coordinates $(x, \theta)$ on $SM$ where $\theta$ is the angle between $v$ and $\partial/\partial x_1$. In these coordinates $V = \partial/\partial \theta$ and $\eta_+$ and $\eta_-$ are $\partial$ and $\overline{\partial}$ type operators, see \cite[Appendix B]{PSU_hd}.

For any $m \in \mZ$ we define 
$$
\Lambda_m = \{ u \in C^{\infty}(SM) \,;\, Vu = imu \}.
$$
In the $(x, \theta)$ coordinates elements of $\Lambda_m$ look locally like $h(x) e^{im\theta}$. Spherical harmonics may be further decomposed as 
\begin{gather*}
\Omega_0 = \Lambda_0, \\
\Omega_m = \Lambda_m \oplus \Lambda_{-m} \text{ for } m \geq 1.
\end{gather*}
Any $u \in C^{\infty}(SM)$ has a decomposition $u = \sum_{m=-\infty}^{\infty} u_m$ where $u_m \in \Lambda_m$. The geodesic vector field decomposes as 
$$
X = \eta_+ + \eta_-
$$
where $\eta_{\pm}: \Lambda_m \to \Lambda_{m \pm 1}$.

Let now $\mc{E}$ be a Hermitian bundle of rank $n$ over $M$, and let $\nabla^{\mc{E}}$ be a Hermitian connection on $\mc{E}$. As in Section \ref{sec_pestov}, we denote by $\mc{E}$ and $\nabla^{\mc{E}}$ the pullback bundle over $SM$ and the pullback connection, and we have the operator $\mathbb{X}$ as before. We wish to discuss the analogues of $X_{\perp}$ and $V$. To do this, define the linear operator $G: C^{\infty}(SM ; N \otimes \mc{E}) \to C^{\infty}(SM ; \mc{E})$ by requiring that 
\[
G(Z \otimes u)(x,v) = \langle Z(x,v), iv \rangle u(x,v), \qquad Z \in C^{\infty}(SM ; N), \ \ u \in C^{\infty}(SM ; \mc{E}).
\]
We then define $\mathbb{X}_{\perp}$ and $\mathbb{V}$ acting on $C^{\infty}(SM ; \mc{E})$ by 
\begin{align*}
\mathbb{X}_{\perp} u &:= -G(\hd\,^{\mc{E}} u), \\
\mathbb{V} u &:= G(\vd\,^{\mc{E}} u).
\end{align*}
We also define the twisted Guillemin-Kazhdan operators 
\[
{\bf \mu}_{\pm} := \frac{1}{2}(\mathbb{X}�\pm i \mathbb{X}_{\perp}).
\]
If $U$ is a trivializing neighborhood for $\mc{E}$ and if $(e_1, \ldots, e_n)$ is an orthonormal local frame over $U$, then any $u \in C^{\infty}(SU ; \mc{E})$ is of the form $u = \sum_{k=1}^n u^k e_k$ with $u^k \in C^{\infty}(SU)$ (here we write $e_j$ for $\pi^* e_j$), and $\nabla^{\mc{E}}$ is represented as $d+A$ where $A = (A^k_l)$ is a skew-Hermitian matrix of $1$-forms. Interpreting $1$-forms $a$ as functions on $SM$ by $a(x,v) = a_x(v)$, we have the splitting $A = A_+ + A_-$ where 
\[
(A_{\pm})^k_l = \frac{1}{2}(A^k_l \pm \frac{1}{i} V A^k_l).
\]
Then $A_{\pm}$ is a matrix with entries in $\Lambda_{\pm 1}$, and since $A$ is skew-Hermitian one has $A_{\pm}^* = -A_{\mp}$. One can now check that the above operators have local coordinate representations 
\begin{align*}
\mathbb{X}_{\perp} (\sum_{k=1}^nu^k e_k) &= \sum_{k=1}^n \Big(X_{\perp} u^k - \sum_{l=1}^n(V A^k_l) u^l\Big) e_k, \\
\mathbb{V} (\sum_{k=1}^n u^k e_k) &= \sum_{k=1}^n (Vu^k) e_k, \\
{\bf \mu}_{\pm} \sum_{k=1}^n (u^k e_k) &= \sum_{k=1}^n \Big(\eta_{\pm} u^k + \sum_{l =1}^n(A_{\pm})^k_l u^l\Big ) e_k.
\end{align*}

Setting $\Lambda_m(SM ; \mc{E}) = \{ u \in C^{\infty}(SM ; \mc{E}) \,;\, \mathbb{V} u = im u \}$, any $u \in C^{\infty}(SM ; \mc{E})$ has an $L^2$-orthogonal decomposition 
\[
u = \sum_{m=-\infty}^{\infty}�u_m
\]
where $u_m \in \Lambda_m(SM ; \mc{E})$. The operators ${\bf \mu}_{\pm}$ satisfy ${\bf \mu}_{\pm}: \Lambda_m(SM ; \mc{E}) \to \Lambda_{m \pm 1}(SM ; \mc{E})$, and $\mathbb{X} = {\bf \mu}_+ + {\bf \mu}_-$. The relation to $\mathbb{X}_{\pm}$ is as follows: $\mathbb{X}_+ u_0 = {\bf \mu}_+ u_0 + {\bf \mu}_- u_0$ for $u_0 \in \Lambda_0(SM ; \mc{E})$, and for $m \geq 1$ we have 
\begin{gather*}
\mathbb{X}_+ (u_m + u_{-m}) = {\bf \mu}_+ u_m + {\bf \mu}_- u_{-m}, \\
\mathbb{X}_- (u_m + u_{-m}) ={\bf \mu}_- u_m + {\bf \mu}_+ u_{-m}
\end{gather*}
where $u_j \in \Lambda_j(SM ; \mc{E})$.

Let $\star$ be the Hodge star operator on $(M,g)$. The curvature  $f^{\mc{E}}$ of $\nabla^{\mc{E}}$ is a $2$-form with values in skew-Hermitian endomorphisms of $\mc{E}$. In a trivializing neighborhood $U \subset M$ we may represent $\nabla^{\mc{E}}$ as $d+A$, and then $f^{\mc{E}}$ is represented as $dA + A \wedge A$, an $n \times n$ skew-Hermitian matrix of $2$-forms. Since $d=2$, $i \star f^{\mc{E}}$ is a smooth section on $M$ with values in Hermitian endomorphisms of $\mc{E}$ and thus having real eigenvalues. Denote by $\lambda_1 \leq \ldots \leq \lambda_n$ the eigenvalues of $i \star f^{\mc{E}}$. Since the ordered eigenvalues of a Hermitian matrix are Lipschitz continuous functions of its entries (see e.g.\ \cite[Section 1.3.3]{Tao_randommatrix}), the maps $\lambda_j: M \to \mR$ are Lipschitz continuous.

Finally, we recall the commutator formula on $C^{\infty}(SU ; \C^n)$ (see \cite[Lemma 4.3]{Pa2}),
\[
[\eta_+ + A_+, \eta_- + A_-] = \frac{i}{2} (KV + \star f_A )
\]
where $f_A = dA + A \wedge A$. This implies a corresponding formula on $C^{\infty}(SM ; \mc{E})$:
\begin{equation} \label{bfmu_commutator_formula}
[{\bf \mu}_+, {\bf \mu}_-] =  \frac{i}{2} (K \mathbb{V} + \star f^{\mc{E}} ).
\end{equation}

After these preliminaries, we state the result ensuring absence of nontrivial twisted CKTs (Theorem \ref{main_thm_twistedckts_twodim} is part (c) below). Here $\chi(M)$ is the Euler characteristic of $M$.

\begin{Theorem}
Let $(M,g)$ be a closed oriented Riemannian surface, let $\mc{E}$ be a Hermitian bundle of rank $n$ over $M$, and let $\nabla^{\mc{E}}$ be a unitary connection on $\mc{E}$. Denote by $\lambda_1 \leq \ldots \leq \lambda_n$ the eigenvalues of $i \star f^{\mc{E}}$ counted with multiplicity.
\begin{enumerate}
\item[(a)] 
If $m \in \Z$ and if 
$$
\int_M \lambda_1 \,dV > 2\pi m \chi(M),
$$
then any $u \in \Lambda_m(SM ; \mc{E})$ satisfying ${\bf \mu}_+ u = 0$ must be identically zero.
\item[(b)] 
If $m \in \Z$ and if 
$$
\int_M \lambda_n \,dV < -2\pi m \chi(M),
$$
then any $u \in \Lambda_{-m}(SM ; \mc{E})$ satisfying ${\bf \mu}_- u = 0$ must be identically zero.
\item[(c)] 
If $m \geq 1$ and if 
$$
2\pi m \chi(M) < \int_M \lambda_1 \,dV \ \ \text{ and } \quad \int_M \lambda_n \,dV < -2\pi m \chi(M),
$$
then any $u \in \Omega_m(SM ; \mc{E})$ satisfying $\mathbb{X}_+ u = 0$ must be identically zero.
\end{enumerate}
\end{Theorem}

A few remarks are in order:
\begin{enumerate}
\item[1.] 
The condition for $\lambda_1$ is sharp: the work \cite{Pa2} furnishes examples of connections that admit nontrivial twisted CKTs with $m=1$ and satisfy $\lambda_1=K$, so that $\int_M \lambda_1 \,dV = 2\pi \chi(M)$ by the Gauss-Bonnet theorem.
\item[2.]
The condition for $\lambda_1$ is conformally invariant: if $c$ is a positive function, then $\star_{cg} f^{\mc{E}} = c^{-1} \star_g f^{\mc{E}}$ and so $\lambda_{1,cg} \,dV_{cg} = \lambda_1 \,dV_g$.
\item[3.] 
If $\mc{E} = M \times \C$ is the trivial line bundle and $\nabla^{\mc{E}}$ is any Hermitian connection, then $\nabla^{\mc{E}} = d + A$ for some purely imaginary scalar $1$-form $A$, and $\lambda_1 = i \star f^{\mc{E}} = i \star dA$ and 
\[
\int_M \lambda_1 \,dV = i \int_M (\star dA) \,dV = i \int_M \star \star dA = i \int_M dA = 0.
\]
In particular, if $M$ has genus $\geq 2$ and if $\nabla^{\mc{E}}$ is any Hermitian connection on the trivial line bundle, then $\mathbb{X}_+$ has trivial kernel on $\Omega_m$ for all $m \geq 1$.
\end{enumerate}

\begin{proof}
We only prove (a), since (b) is analogous and (c) follows by combining (a) and (b). Given the condition on $\lambda_1$, we will prove a Carleman estimate 
$$
\norm{e^{-\varphi} w} \leq C \norm{e^{-\varphi} {\bf \mu}_+ w}, \qquad w \in \Lambda_m(SM ; \mc{E}),
$$
where $\norm{\,\cdot\,}$ is the norm on $L^2(SM ; \mc{E})$ and $\varphi$ is a Carleman weight, that is, a suitable real valued function in $C^{\infty}(M)$ such that the $L^2$ norm of $e^{-\varphi} w$ can be controlled by the $L^2$ norm of $e^{-\varphi} {\bf \mu}_+ w$. If $u$ satisfies ${\bf \mu}_+ u = 0$, taking $w=u$ in this estimate gives $u=0$ as required.

To prove the Carleman estimate, let $\varphi \in C^{\infty}(M)$ be real valued and consider the conjugated operator 
$$
P = e^{-\varphi} \circ {\bf \mu}_+ \circ e^{\varphi} = {\bf \mu}_+ + (\eta_+ \varphi).
$$
Here, we write $\eta_+ \varphi$ instead of $(\eta_+ \varphi) \mathrm{Id}$ etc. The $L^2$ adjoint of $P$ is 
$P^* = -{\bf \mu}_- + (\eta_- \varphi)$, and integration by parts (where $(\,\cdot\,,\,\cdot\,)$ is the $L^2(SM ; \mc{E})$ inner product) yields 
\begin{equation} \label{p_padjoint_identity}
\norm{Pw}^2 = \norm{P^* w}^2 + ([P^*, P]w, w).
\end{equation}
The commutator is given by 
$$
[P^*, P]w = [{\bf \mu}_+, {\bf \mu}_-]w - (\eta_+ \eta_- \varphi + \eta_- \eta_+ \varphi) w .
$$
Using the formula \eqref{bfmu_commutator_formula} and the fact that $(\eta_+ \eta_- + \eta_- \eta_+) \varphi = \frac{1}{2} \Delta_g \varphi$ since $\varphi \in C^{\infty}(M)$, where $\Delta_g$ is the (negative) Laplace-Beltrami operator on $(M,g)$, it follows that 
$$
[P^*, P]w = \frac{1}{2} (-\Delta_g \varphi - mK + i \star f^{\mc{E}}) w, \qquad w \in \Lambda_m.
$$
Suppose that we can find $\varphi \in C^{\infty}(M)$ such that for some constant $c > 0$, 
\begin{equation} \label{deltavarphi_fa_condition}
-\Delta_g \varphi - mK + i \star f^{\mc{E}} \geq c\, \mathrm{Id} \quad \text{on }M
\end{equation}
as positive definite endomorphisms. Then the commutator term in \eqref{p_padjoint_identity} is positive and satisfies  $([P^*, P]w, w) \geq \frac{c}{2} \norm{w}^2$, so it follows that 
$$
\frac{c}{2} \norm{w}^2 \leq \norm{e^{-\varphi} \mu_+ (e^{\varphi} w)}^2, \qquad w \in \Lambda_m.
$$
This gives the desired Carleman estimate upon replacing $w$ by $e^{-\varphi} w$.

It remains to find $\varphi$ with the property \eqref{deltavarphi_fa_condition}. To do this, we choose a real valued function $f \in C^{\infty}(M)$ satisfying the following two conditions: 
\begin{gather*}
f + \lambda_1 > 0 \ \text{on }M, \\
\int_M f \,dV = -2\pi m \chi(M).
\end{gather*}
If $f$ satisfies these, then $\int_M (mK + f) \,dV = 0$ by the Gauss-Bonnet theorem and thus there exists a solution $\varphi$ of the equation 
$$
-\Delta_g \varphi = mK + f \quad \text{in }M.
$$
This $\varphi$ will satisfy \eqref{deltavarphi_fa_condition} because $-\Delta_g \varphi - mK + i \star f^{\mc{E}} \geq f + \lambda_1 \geq c > 0$ on $M$.

To find $f$, we use the assumption on $\lambda_1$ and define $\eps > 0$ by 
$$
\eps = \frac{1}{\text{Vol}(M)} \left[ \int_M \lambda_1 \,dV - 2\pi m \chi(M) \right].
$$
Since $\lambda_1 \in C(M)$, we can choose $h \in C^{\infty}(M)$ with 
$$
\norm{h-\lambda_1}_{L^{\infty}(M)} \leq \frac{\eps}{4}.
$$
We then define $f$ as $f = -h + \eps_0$, where $\eps_0$ is the constant determined by 
$$
\int_M f \,dV = -2\pi m \chi(M).
$$
Thus we have 
\begin{align*}
\eps_0 &= \frac{1}{\text{Vol}(M)} \left[ \int_M h \,dV - 2\pi m \chi(M) \right] \\
 &= \frac{1}{\text{Vol}(M)} \left[ \int_M \lambda_1 \,dV + \int_M (h-\lambda_1) \,dV - 2\pi m \chi(M) \right] \\
 &\geq \eps - \norm{h-\lambda_1}_{L^{\infty}(M)} \geq \frac{3\eps}{4}.
\end{align*}
It follows that $f + \lambda_1 = \lambda_1 - h + \eps_0 \geq \frac{3\eps}{4} - \frac{\eps}{4} = \frac{\eps}{2}$, so $f$ satisfies the two required conditions. This concludes the proof.
\end{proof}

\section{Transparent pairs} \label{sec_transparent_pairs}

In this final section we consider the problem of when the parallel transport associated with a pair
$(\nabla^{\mc{E}},\Phi)$ determines the pair up to gauge equivalence in the case of closed manifolds.
This problem is discussed in detail in \cite{Pa2,Pa3,P1,P2}, but the results are mostly for
$d=2$.

Since there is no boundary, we need to consider the parallel transport of a pair along closed geodesics.
We shall consider a simplified version of the problem, which is interesting in its own right.
The bundle $\mathcal E$ will be trivial (hence $\nabla^{\mc{E}}=d+A$) and we will attempt to understand those pairs
$(A,\Phi)$ with the property that the parallel transport along closed geodesics is the identity.
These pairs will be called {\it transparent} as they are invisible from the point of view of the closed geodesics of the Riemannian metric.

Let $(M,g)$ be a closed Riemannian manifold, $A$ a unitary connection and $\Phi$ a skew-Hermitian Higgs field.
The pair $(A,\Phi)$ naturally induces a cocycle over
the geodesic flow $\varphi_t$ of the metric $g$ acting on the unit sphere bundle
$SM$ with projection $\pi:SM\to M$. The cocycle takes values in the group $U(n)$ and
is defined as follows: let
$C:SM\times \re\to U(n)$ be determined by
\[\frac{d}{dt}C(x,v,t)=-(A(\varphi_{t}(x,v))+\Phi(\pi\circ\varphi_{t}(x,v)))C(x,v,t),\;\;\;\;\;C(x,v,0)=\mbox{\rm Id}.\]
The function $C$ is a {\it cocycle}: 
\[C(x,v,t+s)=C(\varphi_{t}(x,v),s)\,C(x,v,t)\]
for all $(x,v)\in SM$ and $s,t\in\re$. The cocycle $C$ is said to be {\it cohomologically trivial}
if there exists a smooth function $u:SM\to U(n)$ such that
\[C(x,v,t)=u(\varphi_{t}(x,v))u^{-1}(x,v)\]
for all $(x,v)\in SM$ and $t\in\re$. We call $u$ a trivializing function and note that two trivializing functions $u_{1}$ and $u_{2}$ (for the same cocycle) are related by $u_{2}w=u_{1}$
where $w:SM\to U(n)$ is constant along the orbits of the geodesic flow. In particular, if $\varphi_{t}$ is transitive (i.e. there is a dense orbit) there is a unique trivializing function up to
right multiplication by a constant matrix in $U(n)$.

\begin{Definition} {\rm We will say that a pair $(A,\Phi)$ is cohomologically trivial if $C$ is cohomologically trivial.
The pair $(A,\Phi)$
is said to be {\it transparent} if $C(x,v,T)=\mbox{\rm Id}$ every time
that $\varphi_{T}(x,v)=(x,v)$. }
\end{Definition}

Observe that the gauge group given by the set of smooth maps
$r:M\to U(n)$ acts on pairs as follows:
\[(A,\Phi)\mapsto (r^{-1}dr+r^{-1}Ar,r^{-1}\Phi r).\]
This action leaves invariant the set of cohomologically trivial pairs: indeed,
if $u$ trivializes the cocycle $C$ of a pair $(A,\Phi)$, then it is easy to check that $r^{-1}u$ trivializes the cocycle of the pair $(r^{-1}dr+r^{-1}Ar,r^{-1}\Phi r)$.

Obviously a cohomologically trivial pair is transparent.
There is one important situation in which both notions agree. If $\varphi_t$ is
Anosov, then the Livsic theorem \cite{L1,L2} together with the regularity results in \cite{NT}
imply that a transparent pair is also cohomologically trivial. 
We already pointed out that the  Anosov property is satisfied, if for example $(M,g)$ has negative curvature.

Given a cohomologically trivial pair $(A,\Phi)$, a trivializing function $u$ satisfies
\begin{equation}
(X+A+\Phi)u=0.
\label{eq:transparent}
\end{equation}
If we assume now that $(M,g)$ is negatively curved and there are no nontrivial CKTs, then Theorem \ref{main_thm_closed_finitedegree} implies that
$u=u_0$. If we split equation (\ref{eq:transparent}) in degrees zero and one we obtain $\Phi u_0=0$ and
$du+Au=0$. Equivalently, $\Phi=0$ and $A$ is gauge equivalent to the trivial connection.
Hence we have proved

\begin{Theorem} \label{main_thm_transparent_pairs}
Let $(M,g)$ be a closed negatively curved manifold and $(A,\Phi)$ a transparent pair.
If there are no nontrivial twisted CKTs, then $A$ is gauge equivalent to the trivial connection and
$\Phi=0$.
\label{thm:transparentpairs}
\end{Theorem}

In analogy with Theorem \ref{mainthm_boundary2} we could also consider two pairs $(A,\Phi)$ and $(B,\Psi)$ and 
a theorem in this direction is also possible along the lines of \cite[Section 6]{P2}. However in order to shorten the exposition, we will not discuss this case here.

\bibliographystyle{alpha}

\end{document}